\documentclass[12pt]{article}

\usepackage{amsfonts}
\usepackage{amssymb}
\usepackage{amsmath}
\usepackage{amsthm}
\usepackage{mathptmx}
\usepackage[latin1]{inputenc}
\usepackage{txfonts}
\usepackage{mathrsfs}
\numberwithin{equation}{section}

\usepackage{a4}
\usepackage{mathrsfs}
\usepackage{fullpage}
\usepackage{color}
\usepackage[driverfallback=pdfmx,,colorlinks,bookmarksopen,bookmarksnumbered,citecolor=blue, urlcolor=blue]{hyperref}
\newenvironment{keywords}{
\list{}{\advance\topsep by0.35cm\relax\small
\leftmargin=1cm
\labelwidth=0.35cm
\listparindent=0.35cm
 \itemindent\listparindent
 \rightmargin\leftmargin}\item[\hskip\labelsep
 \bfseries Keywords:]}
 {\endlist}

\newtheorem{theorem}{Theorem}[section]
\newtheorem{lemma}[theorem]{Lemma}
\newtheorem{remark}[theorem]{Remark}
\newtheorem{pro}[theorem]{Proposition}

\newcommand{\N}{\mathbb{N}}

\newcommand{\A}{\mathscr{A}}
\newcommand{\B}{\mathscr{B}}

\newcommand{\Ar}{\sqrt{A}}
\newcommand{\He}{\mathcal{H}}

\begin{document}

\title{Well-posedness and exponential stability for abstract evolution equations with delay in the nonlinear source: frictional and viscoelastic cases}
 \date{}
\author{
  \textsc{Houria Chellaoua}\thanks{
    Laboratory of Pure and Applied Mathematics, University of Laghouat, P.O. BOX 37G, Laghouat (03000), Algeria; 
    Department of Mathematics and Computer Science, Faculty of Science and Technology, University of Ghardaia, Ghardaia (47000), Algeria. 
    E-mail: chellaoua.houria@univ-ghardaia.edu.dz
  }
  \and
  \textsc{Yamna Boukhatem}\thanks{
  	Laboratory of Innovation in Mathematics and Applied Sciences (LIMAS),
  	National Higher School of Mathematics,
  	Scientific and Technology Hub of Sidi Abdellah,
  	P.O. Box 75, Algiers 16093, Algeria.
    E-mail: yamna.boukhatem@nhsm.edu.dz
  }
  \and
  \textsc{Cristina Pignotti}\thanks{
    Dipartimento di Ingegneria e Scienze dell'Informazione e Matematica, Universit\`a dell'Aquila, Via Vetoio, Loc. Coppito, 67100 L'Aquila, Italy. 
    E-mail: cristina.pignotti@univaq.it
  }
}
\maketitle

\begin{abstract}
In this paper, we establish the well-posedness and exponential stability of a class of semilinear neutral abstract evolution equations with constant time delay. Two damping mechanisms are considered: frictional damping of the form $CC^*u_t$ and viscoelastic damping with memory. Under suitable assumptions, we prove global well-posedness and exponential decay of solutions for sufficiently small initial data. This result is of considerable importance, as it provides a unified framework for the analysis of a wide class of neutral systems arising in structural mechanics, seismic isolation, and control theory. The abstract results are illustrated by some concrete examples. The analysis relies on a combination of semigroup theory, energy estimates, Duhamel's formula, and  Gronwall-type arguments. 
\end{abstract}


\section{Introduction}

 Let $H$ be a real Hilbert space with inner product and related norm denoted by $ \langle . , . \rangle $ and $ \Vert \: . \: \Vert$, respectively. Let $A : D(A)\longrightarrow H $ be a self-adjoint linear positive operator, with dense domain $D(A)\subset H,$ and $\tau >0$ represents a time delay.
In this work, we consider firstly the following second-order abstract evolution equation
 \begin{equation}\label{p}
\left\{
\begin{array}{ll}
u_{tt}(t)+A u(t)+ CC^*u_t(t)+\alpha(t) BB^* u_t(t-\tau)=\nabla \psi_1(u) + \nabla \psi_2(u(t-\tau)), & \quad t \in(0,+\infty), \\
u(t-\tau)=f_{0}(t-\tau),
 \; u_{t}(t-\tau)=g_{0}(t-\tau)& \quad t \in [0,\tau],\\
\end{array}
\right.
\end{equation}
where the initial data $(f_0,g_0)$ are taken in suitable spaces and
we denote $u_0 := f_0(0)$ and $u_1 := g_0(0)$. $B$ is a bounded linear operator of $H$ into itself. $\alpha$ is an $L^1_{loc}([0,+\infty))$ function satisfying, for all $t\geq 0$
\begin{equation}\label{cdt_alpha}
 \int_{t-\tau}^{t}\vert \alpha(s)\vert ds\leq \sigma \quad \text{for some} \quad \sigma>0.
\end{equation}

Moreover, for a given real Hilbert space $W,$ that will be identified with its dual space, $C:W\to H$ is a bounded linear operator. We denote ${C}^*, B^*$ the adjoint of ${C}, B$ respectively. We assume that the damping operator ${C}{C}^*$ satisfies a control geometric property (see e.g. \cite{Bardos} or \cite[Chapter 5]{K}), namely, for the model 
\begin{equation}\label{modello_linear}
\begin{array}{l}
\displaystyle{u_{tt}(t)+A u(t)+{C}{C}^*u_t(t)=0,
}\\
\displaystyle{u(0)=u_0, }\ \ {\displaystyle{u_t(0)=u_1, }}
\end{array}
\end{equation}
the observability inequality 
\begin{equation}\label{OI}
\int_0^T\Vert {C}^*u_t(t)\Vert_{W}^2 dt \ge c \left( \|u_1\|_H^2+ \|A^{\frac 12}u_0\|_H^2\right ),
\end{equation}
holds for some time $T>0,$ for a suitable positive constant $c$
independent of the initial datum $(u_0, u_1).$
The delay operator $B$ can be, instead, any bounded operator.

{Time delays occur naturally in numerous applications and physical models, where even an arbitrarily small delay is known to potentially destabilize an otherwise uniformly exponentially stable system (cf. \cite{DLP, nicaise_2006, XYL}). Nonetheless, stability can be recovered through an appropriate choice of the delay value (cf. \cite{Gugat}) or through suitable feedback laws (cf. \cite{nicaise_2006, XYL}). Since a time delay can act as a source of instability, the construction of suitable stabilizing controls, which are considerably more delicate to handle than the standard ones, becomes crucial.}

Inspired by \cite{alabau_2008}, we assume the following assumptions

\textbf{(A1)} For every \( u \in D(\Ar) \), there exists a constant \( c(u) > 0 \) such that
\[
|D\psi_i(u)(v)| \leq c(u)\|v\| \quad \forall v \in D(\Ar), \quad i=1, 2.
\]
Then, for $i=1,2$, \( \psi_i \) can be extended to the whole \( H \) and we denote by \( \nabla \psi_i(u) \) the unique vector representing \( D\psi_i(u) \) in the Riesz isomorphism, i.e.,
\[
\langle \nabla \psi_i(u), v \rangle = D\psi_i(u)(v), \quad \forall v \in H.
\]

\textbf{(A2)} For all 
\( R> 0 \), there exists a constant \( L(R) > 0 \) such that
\[
\|\nabla \psi_i(u) - \nabla \psi_i(v)\| \leq L(R)\|\Ar (u - v)\|,
\]
 for all \( u, v \in D(\Ar ) \) satisfying \( \|\Ar u\| \leq R \) and \( \|\Ar v\|\leq R\), for $i=1,2,$ and $L(R) \to 0$, if $R \to 0$.

\textbf{(A3)} $\psi_i(0)=\nabla\psi_i(0)=0$, for $i=1,2,$ and there exist a strictly increasing continuous function \( h_1 \), such that
\[\vert \langle 
\nabla \psi_1(u), u\rangle\vert\leq h_1(\|\Ar u\|) \|\Ar u\|^2, \quad \forall u \in D(\Ar ).
\]

Next, as a different class of systems, we consider a model with delay feedback and nonlinearities as above, but involving a memory term in place of the linear dissipative term, that is
 \begin{equation}\label{visco_p}
\left\{
\begin{array}{ll}
u_{tt}(t)+A u(t)-\int_{0}^{+\infty} k(s)Au(t-s)ds+\alpha(t) BB^* u_t(t-\tau)=\nabla \psi_1(u) + \nabla \psi_2(u(t-\tau)), & \quad t \in(0,+\infty), \\
u(t-\tau)=f_{0}(t-\tau), & \quad t\in (-\infty, \tau],\\
u_{t}(t-\tau)=g_{0}(t-\tau)& \quad t \in [0,\tau],\\
\end{array}
\right.
\end{equation}
where  the operators $A$ and $B$, the function $\alpha$, and the nonlinear functions $\psi_i$, for $i=1,2$, all satisfy the same conditions as before, and the initial data $(f_0,g_0)$ are taken in suitable spaces. Also here, we denote $u_0 := f_0(0)$ and $u_1 := g_0(0)$. The memory kernel \(k:[0,+\infty)\to[0,+\infty)\) satisfies the following classical assumptions:
\begin{enumerate}
 \item \(k\in C^1(\mathbb{R}_+)\cap L^1(\mathbb{R}_+);\)
 \item \(k(0)=k_0>0;\)
 \item \(\displaystyle \int_0^{\infty} k(t)\,\mathrm{d}t =: \tilde{k} < 1;\)
 \item $k'(t)\le -\zeta\,k(t)$, for some $\zeta>0.$
\end{enumerate}

The main goal of this paper is to establish the well-posedness of systems \eqref{p} and \eqref{visco_p} and to derive an exponential decay estimate for solutions corresponding to  sufficiently small initial data, subject to appropriate conditions on the parameters of the models. A linear model of this type of abstract evolution equation has been analyzed in \cite{nicaise_2015}. The authors showed that for that system, which is exponentially
stable in absence of time delay, the exponential stability is preserved
if the norm of the delay feedback operator  is sufficiently small. This result has then been extended to abstract semilinear evolution equations, see \cite{nicaise_2018,komornik_2022}.
See also \cite{chellaoua_2021, nicaise_2006,
 	Paolucci_2022, continelli} for further developments and extensions. 
 	
 	  For recent stability results for other specific models including time delay effects we refer to \cite{Oliveira, Baudouin, Ma, Capistrano, Issa, Ammari}.

Concerning viscoelastic models without delay, a one-dimensional viscoelastic wave equation was studied in \cite{dafermos_1970} where the author established that the energy of the problem decays asymptotically to zero under  Dirichlet boundary conditions and some assumption on the memory kernel. Since then, numerous papers have extended these results to various settings; for example exponential stability  is obtained in \cite{Giorgi} for exponentially decaying memory kernels. For results including both memory damping and time delays we mention
 \cite{alabau_2014,guesmia_2013,pignotti_2017,kirane_2011,Dai_2014,feng_2017, paolucci_2021_visco, continelli}. 

Motivated by these works, we are interested in this paper in studying a class of semilinear abstract neutral evolution equations with frictional or viscoelastic damping in the presence of constant delay with time-varying coefficient. 

 {The main novelty of this paper is the introduction of a time delay in the nonlinear source term. From a mathematical point of view, this feature creates additional difficulties. In particular, the energy functionals commonly used in the absence of delay in the source are no longer suitable, and new functionals have to be introduced. Moreover, the delayed source gives rise to an additional term in the energy estimates, which requires a careful treatment and suitable control. This extension is also relevant from the point of view of applications. Indeed, when the source term is artificially generated, time delays naturally arise in the practical implementation of control strategies and should therefore be incorporated into the mathematical model.}

Through a combination of semigroup arguments, suitable energy estimates, and an iterative method, we prove, under suitable assumptions, the well-posedness of the problems and  exponential decay estimates for solutions corresponding to small initial data. This work therefore extends the analysis initiated in \cite{nicaise_2015, komornik_2022} and further developed in \cite{paolucci_2021_visco,Paolucci_2022, continelli}. 

The remainder of this paper is structured as follows. Section \ref{linear_dam_sec} is devoted to introducing the abstract framework for systems with frictional damping of the form $CC^*u_t$, establishing the well-posedness and exponential stability results for system \eqref{p}. In Section \ref{visco_dam_sec}, we study abstract viscoelastic systems, where the damping is of memory type. For this second class of systems, we prove analogous well-posedness and stability results under suitable assumptions on the memory kernel and the nonlinearities. Finally, in Section \ref{expls_sec}, we illustrate the applicability of our abstract results by presenting concrete examples.

\section{Abstract systems with frictional damping: system \eqref{p}} \label{linear_dam_sec}
This section is devoted to the abstract reformulation of problem \eqref{p}, followed by well-posedness and exponential stability results. Let us reformulate problem \eqref{p} into the abstract form 
\begin{equation}\label{P_abstract}
\left\{
\begin{array}{ll}
U_t(t) = \A U(t) +\alpha(t)\B U(t - \tau)+ F(U(t), U(t - \tau)), &\quad t \in (0, +\infty), \\
 U(t - \tau)=\Phi(t-\tau), &\quad t \in [0, \tau]
\end{array}
\right.
\end{equation}
where $U(t) = (u(t),v(t) )^T$ with $v(t) = u_t(t)$. Let us denote $U_0=U(0).$ We introduce the Hilbert space $\mathcal{H}=D(\sqrt{A})\times H$, equipped with the inner product
$$\left\langle \begin{pmatrix} u_1 \\ v_1 \end{pmatrix},\begin{pmatrix} u_2 \\ v_2 \end{pmatrix}\right\rangle_{\mathcal{H}}=\left \langle \Ar u_1, \Ar u_2\right\rangle +\langle v_1, v_2\rangle .$$
The linear operators $\A$ and $\B$ are given by
$$\A \begin{pmatrix} u \\ v \end{pmatrix} = \begin{pmatrix} v \\ -A u-CC^*v \end{pmatrix}, \quad \B \begin{pmatrix} u \\ v \end{pmatrix} = \begin{pmatrix} 0 \\ -BB^*v \end{pmatrix} $$$$ \quad
\text{and} 
\quad F(U(t), U(t - \tau)) = \begin{pmatrix}
0 \\
\nabla \psi_1(u(t)) + \nabla \psi_2(u(t - \tau))
\end{pmatrix}.$$

We know that, under some controllability assumptions on the damping operator ${C}{C}^*,$ namely when the observability inequality \eqref{OI} holds for the linear model \eqref{modello_linear}, the operator $\A$ generates an exponentially stable $C_0$-semigroup $\{ S(t)\}_{t\geq 0}$, specifically there exist $M,\omega >0$ such that
\begin{equation}\label{decay_semigroup}
\|S(t)\|_{\mathcal L(\mathcal H)} \leq Me^{-\omega t}, \quad \forall t\geq 0.
\end{equation}
Moreover, the previous hypotheses (A2)-(A3) on $\psi_i, i=1,2,$ imply the following properties of $F$:
\begin{itemize}
\item[(F1)] $F(0)=0$;
\item[(F2)] for any $R>0$ there exists a constant $L(R)>0$ such that
\begin{equation}\label{F_local}
 \| F(U_1, U_1^\tau) - F(U_2, U_2^\tau) \| 
\leq L(R) \left( \| U_1 - U_2 \| + \| U_1^\tau - U_2^\tau \| \right),
\end{equation}
for $\| U_i\|,\| U_i^\tau \|\leq R,$ $i=1,2,$ with $L(R) \to 0$, if $R \to 0$, where $U^\tau=U(t-\tau)$.
\end{itemize}

We first give a local existence result.
\begin{lemma}\label{lem_abst_exis}
 Let us consider the system \eqref{P_abstract} with initial data $\Phi \in C([-\tau, 0]; \He)$. Then, there exists a unique continuous local solution $U$ defined on
a time interval $[0, T)$, with $T >0,$ which is given by
\begin{equation}\label{formule_U_rho}
U(t) = S(t) U_0 + \int_0^t S(t -s) \bigl[\alpha(s)\B U(s - \tau)+ F(U(s), U(s - \tau))\bigr] \, ds,
\end{equation}

for all $t \in [0, T)$.
\end{lemma}
\begin{proof}
By considering $t \in [0,\tau],$ problem \eqref{P_abstract} can be rewritten as a nondelayed inhomogeneous Cauchy problem. Then, by applying the classical theory of nonlinear semigroups, see \cite{pazy}
we obtain the existence of a unique local solution on a set $[0, \delta )$, with $\delta \leq \tau,$ that can eventually exists on a larger interval $[0,T).$
\end{proof}

\subsection{An exponential decay estimate for system \eqref{p} }

For our decay estimate, we need the following assumption on the linear delay feedback.

For suitable constants $\gamma\geq 0$ and $\omega^\prime\in [0,\omega[$, it holds that
 \begin{equation}\label{cdt_alpha_2}
 M b e^{\omega \tau}\int_{0}^{t}\vert \alpha(s+\tau)\vert ds\leq \gamma+\omega^\prime t,\quad \forall t>0, 
\end{equation}
where $b=\Vert\B\Vert_{\mathcal{H}}=\Vert BB^*\Vert$ and $M, \omega$ are the parameters in \eqref{decay_semigroup}.

\begin{theorem}\label{th_stability}
Assume \eqref{cdt_alpha_2}.
 Let $U$ be a solution to \eqref{P_abstract}, defined on a time interval $[0, T'),$ 
satisfying
\begin{equation}\label{U_rho_th}
 \Vert U(t)\Vert \leq C' \quad \forall t\in [-\tau, T'),
\end{equation}
for some $C'>0$ such that 
\begin{equation}\label{serve}
L(C')<\frac {\omega-\omega^\prime}{M(1+e^{\omega\tau})}.
\end{equation}

Then, $U$ decays exponentially; 
 \begin{equation}\label{stab_estimate_2}
 \left\Vert U(t)\right\Vert_{\He} \leq Me^\gamma\left(\left\Vert U_0\right\Vert_{\He}+\beta+\int_{0}^{\tau} e^{\omega s}\vert \alpha(s)\vert\cdot\Vert \Phi(s-\tau)\Vert_{\He} ds\right) e^{-(\omega-\omega^\prime-ML(C')(1+e^{\omega \tau}))t},
\end{equation}
for all $t\in [0,T'),$ where
\begin{equation}\label{beta_eq}
	\beta = L(C')\int_0^\tau e^{\omega s} \Vert \Phi(s-\tau)\Vert_{\mathcal{H}}ds. 
	\end{equation}
\end{theorem}
\begin{remark}\label{remark_1}
 In the following, we will refer to assumption \eqref{U_rho_th} in Theorem \ref{th_stability}, with $C'$ satisfying \eqref{serve}, as a well-posedness assumption.
\end{remark}
\begin{proof} From Duhamel's formula, we obtain 
\begin{eqnarray*}
 \left\Vert U(t)\right\Vert_{\He} &\leq& M e^{-\omega t} \left\Vert U_0\right\Vert_{\He}+ M e^{-\omega t} \int_{0}^{t} e^{\omega s}\vert \alpha(s)\vert\Vert \B U(s-\tau)\Vert_{\He} ds\\
 && + M e^{-\omega t}\int_{0}^{t} e^{\omega s}\Vert F(U(s), U(s- \tau))\Vert_{\He} ds.
\end{eqnarray*} 
From \eqref{F_local} and by using the fact that $F(0,0)=0$, we have
\begin{equation*}
 \Vert F(U(t), U(t- \tau))\Vert_{\He}\leq L(C')\left(\Vert U(t)\Vert_{\He}+\Vert U(t- \tau))\Vert_{\He}\right).
\end{equation*}
Consequently,
\begin{eqnarray*}
 \left\Vert U(t)\right\Vert_{\He} &\leq& M e^{-\omega t} \left\Vert U_0\right\Vert_{\He}+ M e^{-\omega t} \int_{0}^{t} e^{\omega s}\vert \alpha(s)\vert\Vert \B U(s-\tau)\Vert_{\He} ds\\
 && + M L(C')e^{-\omega t}\int_{0}^{t} e^{\omega s}\left(\Vert U(s)\Vert_{\He}+\Vert U(s- \tau))\Vert_{\He}\right)ds,
\end{eqnarray*}
from which it follows that
\begin{eqnarray*}
 \left\Vert U(t)\right\Vert_{\He} &\leq& M e^{-\omega t} \left\Vert U_0\right\Vert_{\He}+ M e^{-\omega t} \int_{0}^{\tau} e^{\omega s}b\vert \alpha(s)\vert \cdot\Vert \Phi(s-\tau)\Vert_{\He} ds
\\
 & & +M e^{-\omega t} b\int_{\tau}^{t} e^{\omega s}\vert \alpha(s)\vert\cdot\Vert U(s-\tau)\Vert_{\He} ds \\
 & & + M L(C')e^{-\omega t}\int_{0}^{t} e^{\omega s}\left(\Vert U(s)\Vert_{\He}+\Vert U(s- \tau))\Vert_{\He}\right)ds.
\end{eqnarray*}

Using the change of variables $\sigma=s-\tau$, we get
\begin{eqnarray*}
 \int_{\tau}^{t} e^{\omega s}\vert \alpha(s)\vert\cdot\Vert U(s-\tau)\Vert_{\He} ds&=& e^{\omega \tau} \int_{0}^{t-\tau} e^{\omega \sigma}\vert \alpha(\sigma+\tau)\vert\cdot\Vert U(\sigma)\Vert_{\He} d\sigma \\
 &\leq& e^{\omega \tau} \int_{0}^{t} e^{\omega s}\vert \alpha(s+\tau)\vert\cdot\Vert U(s)\Vert_{\He} ds.
\end{eqnarray*}
Thus, 
\begin{eqnarray*}
 \left\Vert U(t)\right\Vert_{\He} &\leq& M e^{-\omega t} \left\Vert U_0\right\Vert_{\He}+ M e^{-\omega t} b\int_{0}^{\tau} e^{\omega s}\vert \alpha(s)\vert\cdot\Vert \Phi(s-\tau)\Vert_{\He} ds\\
 & & +M e^{-\omega t} b e^{\omega \tau} \int_{0}^{t} e^{\omega s}\vert \alpha(s+\tau)\vert\cdot\Vert U(s)\Vert_{\He} ds \\
 & & + M L(C')e^{-\omega t}\int_{0}^{t} e^{\omega s}\left(\Vert U(s)\Vert_{\He}+\Vert U(s- \tau))\Vert_{\He}\right)ds.
\end{eqnarray*}
We put 
\begin{equation}
\label{beta_constant}
 \beta= L(C')\int_{0}^{\tau}e^{\omega s}\Vert U(s- \tau))\Vert_{\He}ds.
\end{equation}
Then, 
\begin{eqnarray*}
 \left\Vert U(t)\right\Vert_{\He} &\leq& M e^{-\omega t} \left( \left\Vert U_0\right\Vert_{\He}+ b\int_{0}^{\tau} e^{\omega s}\vert \alpha(s)\vert\cdot\Vert \Phi(s-\tau)\Vert_{\He} ds\right.\\
 & & +b e^{\omega \tau} \int_{0}^{t} e^{\omega s}\vert \alpha(s+\tau)\vert\cdot\Vert U(s)\Vert_{\He} ds \\
 & & \left.+ \beta + L(C')\int_{0}^{t} e^{\omega s}\Vert U(s)\Vert_{\He}ds+ L(C')\int_{\tau}^{t} e^{\omega s}\Vert U(s- \tau))\Vert_{\He}ds\right).
\end{eqnarray*}
Thus, by the change of variables $\sigma=s-\tau,$
\begin{eqnarray*}
 \left\Vert U(t)\right\Vert_{\He} &\leq& M e^{-\omega t} \left( \left\Vert U_0\right\Vert_{\He}+ b\int_{0}^{\tau} e^{\omega s}\vert \alpha(s)\vert\cdot\Vert \Phi(s-\tau)\Vert_{\He} ds\right.\\
 & & +b e^{\omega \tau} \int_{0}^{t} e^{\omega s}\vert \alpha(s+\tau)\vert\cdot\Vert U(s)\Vert_{\He} ds \\
 & & \left.+ \beta + L(C')\int_{0}^{t} e^{\omega s}\Vert U(s)\Vert_{\He}ds+ L(C')e^{\omega \tau} \int_{0}^{t} e^{\omega s}\Vert U(s))\Vert_{\He}ds\right).
\end{eqnarray*}
Finally, for all $t\in [0, T'),$ we obtain
\begin{eqnarray*}
 \left\Vert U(t)\right\Vert_{\He} &\leq& M e^{-\omega t} \left( \left\Vert U_0\right\Vert_{\He}+\beta+b\int_{0}^{\tau} e^{\omega s}\vert \alpha(s)\vert\cdot\Vert \Phi(s-\tau)\Vert_{\He} ds\right.\\
 & & \left.+ \int_{0}^{t}\left(b e^{\omega \tau} \vert \alpha(s+\tau)\vert +L(C')(1+e^{\omega \tau})\right)e^{\omega s}\Vert U(s)\Vert_{\He}ds\right).
\end{eqnarray*}
Now, let us denote $$\widetilde{u}(t) :=e^{\omega t}\left\Vert U(t)\right\Vert_{\He} , \quad \lambda=\left\Vert U_0\right\Vert_{\He}+\beta+ b\int_{0}^{\tau} e^{\omega s}\vert \alpha(s)\vert\cdot\Vert \Phi(s-\tau)\Vert_{\He} ds,$$ 
$$ \phi(s)= b e^{\omega \tau} \vert \alpha(s+\tau)\vert +L(C')(1+e^{\omega \tau}).$$
Hence 
$$ \widetilde{u}(t)\leq M\lambda+M\int_{0}^{t}\phi(s)\widetilde{u}(s)ds.$$
Then, using Gr\"{o}nwall's estimate and the definitions of $\widetilde{u}$, $\lambda$ and $\phi$, we get
\begin{eqnarray}\label{stab_estimate}\notag
 \left\Vert U(t)\right\Vert_{\He} &\leq & M\left(\left\Vert U_0\right\Vert_{\He}+\beta+ b\int_{0}^{\tau} e^{\omega s}\vert \alpha(s)\vert\cdot\Vert \Phi(s-\tau)\Vert_{\He} ds\right) \times \\
 & & \hspace{3 cm}\times e^{ML(C')(1+e^{\omega \tau})t+M b e^{\omega \tau}\int_{0}^{t}\vert \alpha(s+\tau)\vert ds-\omega t}.
\end{eqnarray}
By using the assumptions \eqref{cdt_alpha_2}, we arrive at
 \begin{equation*}\notag
 \left\Vert U(t)\right\Vert_{\He} \leq Me^\gamma\left(\left\Vert U_0\right\Vert_{\He}+\beta+b\int_{0}^{\tau} e^{\omega s}\vert \alpha(s)\vert\cdot\Vert \Phi(s-\tau)\Vert_{\He} ds\right) e^{-(\omega-\omega^\prime-ML(C')(1+e^{\omega \tau}))t}.
\end{equation*}
since $L(C')$ satisfies $L(C')<\frac {\omega-\omega^\prime}{M(1+e^{\omega\tau})},$ the exponential decay of the solution is guaranteed. 
\end{proof}
\subsection{Global existence and exponential stability for system \eqref{p}}

Establishing the exponential stability of the solutions to system \eqref{p} requires verifying that the well-posedness condition holds, see Remark \ref{remark_1}. For this purpose, we introduce the energy functional related to system \eqref{p} as follows
\begin{equation}\label{function_E}
\begin{array}{l}
\displaystyle{ E(t)=E(u(t))= \frac{1}{2}\left\|u_{t}(t)\right\|^{2}+\frac{1}{2}\left\|\sqrt{A}u(t)\right\|^{2}-\psi_1(u(t))+\frac{1}{2}\int^{t}_{t-\tau}
 \Vert \nabla \psi_2(u(s))\Vert^2 ds} \\
 \hspace{2,7 cm}\displaystyle{
+ \frac{1}{2}\int_{t-\tau}^{t}\vert \alpha(s+\tau)\vert \Vert B^*u_t(s) \Vert^2ds} , 
\end{array}
\end{equation}
for all $t\geq 0$.
Note that,  compared with the one in \cite{Paolucci_2022},
this energy functional contains the additional term $\frac{1}{2}\int^{t}_{t-\tau}
\Vert \nabla \psi_2(u(s))\Vert^2 ds,$ which is  crucial in order to menage the delayed source term. 

To prove our global well-posedness and stability result, we need some preliminary results.

\begin{pro}\label{pro_1}
Let $u$ be a solution to \eqref{p} defined on some interval $[0, T').$
If
\begin{equation}\label{aggiunta}
\Vert \nabla\psi_2(u(t))\Vert \le \Vert\sqrt{A}u(t)\Vert,\ \forall t\in [0, T'),
\end{equation}
and 
 the energy functional verifies
\begin{equation}\label{cdt_energy}
E(t)\geq \frac{1}{4}
\Vert u_t(t)\Vert^2 + \frac{1}{4}\left \Vert\sqrt{A}u(t)\right \Vert^2, \quad \forall t\in [0, T'),
\end{equation}
then,
\begin{equation}\label{estimate_energy}
 E(t)\leq C(t)E(0), \quad \forall t\geq 0 ,
\end{equation}
where
\begin{equation}\label{constant_t}
 C(t)=e^{2\int_{0}^{t}[ 1+ \left(\vert\alpha(s)\vert+\vert\alpha(s+\tau)\vert\right)b] ds}.
\end{equation}
\end{pro}
\begin{proof}
 By multiplying the first equation in system \eqref{p} by $u_t$, we obtain
 \begin{equation*}
 \langle u_{tt}(t), u_t(t) \rangle =\langle -A u(t)-CC^*u_t(t)-\alpha(t) BB^* u_t(t-\tau)+\nabla \psi_1(u(t)) + \nabla \psi_2(u(t-\tau)), u_t(t) \rangle .
 \end{equation*}
 Otherwise, we have
 \begin{eqnarray*}
 E^\prime(t)&=& \langle u_{tt}(t), u_t(t) \rangle+\langle\sqrt{A}u(t), \sqrt{A}u_t(t) \rangle-\langle \nabla\psi_1(u), u_t(t) \rangle+\frac{1}{2}
 \Vert \nabla \psi_2(u(t))\Vert^2\\
& &-\frac{1}{2}
 \Vert \nabla \psi_2(u(t-\tau))\Vert^2 + \frac{1}{2}\vert \alpha(t+\tau)\vert \Vert B^*u_t(t) \Vert^2 -\frac{1}{2}\vert \alpha(t)\vert \Vert B^*u_t(t-\tau) \Vert^2 .
\end{eqnarray*}
By combining the previous estimates, we arrive at
 \begin{eqnarray*}
 E^\prime(t)&=& -\Vert C^* u_t(t)\Vert^2-\alpha(t) \langle BB^* u_t(t-\tau), u_t(t) \rangle+\langle\nabla \psi_2(u(t-\tau)), u_t(t) \rangle +\frac{1}{2}
 \Vert \nabla \psi_2(u(t))\Vert^2\\
& &-\frac{1}{2}
 \Vert \nabla \psi_2(u(t-\tau))\Vert^2 + \frac{1}{2}\vert \alpha(t+\tau)\vert \Vert B^*u_t(t) \Vert^2 -\frac{1}{2}\vert \alpha(t)\vert \Vert B^*u_t(t-\tau) \Vert^2 , 
\end{eqnarray*}
and by using Cauchy-Schwartz inequality, we obtain
 \begin{eqnarray*}
 E^\prime(t)&\leq& \frac{1}{2}
 \Vert u_t(t)\Vert^2 + \frac{1}{2}\Vert \nabla \psi_2(u(t))\Vert^2+\frac{1}{2} \left(\vert\alpha(t)\vert+\vert\alpha(t+\tau)\vert\right) \Vert B^*u_t(t) \Vert^2.
\end{eqnarray*}
By using the assumptions \eqref{aggiunta}, we then obtain 
\begin{equation*}
 E^\prime(t)\leq \frac{1}{2}
 \Vert u_t(t)\Vert^2 + \frac{1}{2} \left\Vert\sqrt{A}u(t)\right\Vert^2+\frac{1}{2} \left(\vert\alpha(t)\vert+\vert\alpha(t+\tau)\vert\right)b \Vert u_t(t) \Vert^2 , 
\end{equation*}
thus,
\begin{eqnarray*}
 E^\prime(t)&\leq& \frac{1}{2}
 \Vert u_t(t)\Vert^2 + \frac{1}{2} \left\Vert\sqrt{A}u(t)\right\Vert^2+\frac{1}{2} \left(\vert\alpha(t)\vert+\vert\alpha(t+\tau)\vert\right)b\Vert u_t(t) \Vert^2 \\
 & & +\frac{1}{2} \left(\vert\alpha(t)\vert+\vert\alpha(t+\tau)\vert\right)b\left\Vert\sqrt{A}u(t)\right\Vert^2.
\end{eqnarray*}
Since $ E$ satisfies \eqref{cdt_energy}, we then achieve
\begin{equation*}
 E^\prime(t)\leq 2 \left[\left(\vert\alpha(t)\vert+\vert\alpha(t+\tau)\vert\right)b+1\right] E(t) , 
\end{equation*}
By simple integration, we get
\begin{equation*}
 E(t)\leq E(0)+ 2 \int_{0}^{t} \left[\left(\vert\alpha(s)\vert+\vert\alpha(s+\tau)\vert\right)b +1\right] E(s)ds , 
\end{equation*}
and, from Gr\"{o}nwall's inequality, \eqref{estimate_energy} is obtained with $C(t)$ as in \eqref{constant_t}.
\end{proof}

\begin{lemma} \label{exis_lemma}
Let \( U(\cdot) \) be a non-zero solution to system \eqref{P_abstract} defined on the interval \( [0, \delta) \) and let $T>\delta$. Then,
\begin{enumerate}
 \item If \( h_1\left(\left\|\sqrt{A} u_0 \right\|\right) < \frac{1}{2} \), then \( E(0) > 0 \). 
 \item If \(h_1\left(2\sqrt{2}C^{1/2}(T)E^{1/2}(0)\right) <\frac{1 }{2} \), and \(L\left(2\sqrt{2}C^{1/2}(T)E^{1/2}(0)\right) <\frac{1 }{2} \), where \( C(\cdot) \) is defined in \eqref{constant_t}, then the following estimates hold for all $t \in [0, \delta)$:
\end{enumerate}
\begin{equation}\label{estimate_psi2}
\Vert \nabla\psi_2(u(t))\Vert \le\frac 12 \left\|\Ar u(t)\right\|, 
\end{equation}

\begin{equation}\label{estimate_energy_existance}
E(t) > \frac{1}{4} \|u_t(t)\|^2 + \frac{1}{4} \left\|\Ar u(t)\right\|^2 +\frac{1}{4}\int^{t}_{t-\tau}
 \Vert \nabla \psi_2(u(s))\Vert^2 ds + \frac{1}{4}\int_{t-\tau}^{t}\vert \alpha(s+\tau)\vert \Vert B^*u_t(s) \Vert^2ds.
\end{equation}
In particular, we arrive at
\begin{equation}\label{U_cdt_E}
E(t) > \frac{1}{4} \|U(t)\|_{\He}^2, 
\end{equation}
for all $t \in [0, \delta)$.
\end{lemma}
\begin{proof}
From the assumption (A3) on the function $\psi_1$, we can write
\begin{equation}\label{cdt_psi_1}
|\psi_1(u)|\leq \int_0^1 |\langle \nabla \psi_1(su), u \rangle| ds 
\leq \|\Ar u\|^2 \int_0^1 s h_1(s\|\Ar u\|) ds
\leq \frac{1}{2} h_1(\|\Ar u\|) \|\Ar u\|^2. 
\end{equation}

Hence, from the assumption $h_1(\|\Ar u_0\|) < \frac{1 }{2}$
and from \eqref{cdt_psi_1}, we have that

\begin{eqnarray*}
E(0)&=& \frac{1}{2} \|u_1\|^2 + \frac{1}{2} \|\Ar u_0\|^2 - \psi_1(u_0) +\frac{1}{2}\int^{0}_{-\tau}
 \Vert \nabla \psi_2(f_0(s))\Vert^2 ds \\
& &+ \frac{1}{2}\int_{-\tau}^{0}\vert \alpha(s+\tau)\vert \Vert B^*g_0(s) \Vert^2ds 
 \nonumber \\
&\geq& \frac{1}{2} \|u_1\|^2 + \frac{1}{2} \|\Ar u_0\|^2 
- \frac{1}{2} h_1(\|\Ar u_0\|) \|\Ar u_0\|^2 +\frac{1}{2}\int^{0}_{-\tau}
 \Vert \nabla \psi_2(f_0(s))\Vert^2 ds \\
& & + \frac{1}{2} \int_{-\tau}^0 |\alpha(s+\tau)| \cdot \|B^* g_0(s)\|^2 ds \nonumber \\
&>& \frac{1}{4} \|u_1\|^2 + \frac{1}{4} \|\Ar u_0\|^2+\frac{1}{4}\int^{0}_{-\tau}
 \Vert \nabla \psi_2(f_0(s))\Vert^2 ds + \frac{1}{4} \int_{-\tau}^0 |\alpha(s+\tau)| \cdot \|B^* g_0(s)\|^2 ds.
\end{eqnarray*}
So, we have 
\begin{equation}\label{E1}
 E(0)>\frac{1}{4} \|u_1\|^2 + \frac{1}{4} \|\Ar u_0\|^2+\frac{1}{4}\int^{0}_{-\tau}
 \Vert \nabla \psi_2(f_0(s))\Vert^2 ds + \frac{1}{4} \int_{-\tau}^0 |\alpha(s+\tau)| \cdot \|B^* g_0(s)\|^2 ds >0,
\end{equation}
since $u$ a non zero solution.

Note also that, from \eqref{E1},
$$
\left\|\Ar u_0\right\| <2 E^{1/2}(0)<2\sqrt{2} C^{1/2}(T) E^{1/2}(0),
$$
being $C(T)>1.$
Then, from the lemma's assumption $L(2\sqrt{2} C^{1/2}(T) E^{1/2}(0))<\frac 12$, and the assumptions (A2)-(A3), it follows that
 \begin{equation}\label{E2}
 \Vert \nabla\psi_2(u_0)\Vert \le L (\|\Ar u_0\| ) \left\|\Ar u_0\right\|\le \frac 12 \left\|\Ar u_0\right\| .
\end{equation}

 To establish the second claim, we proceed by contradiction. Denote
\begin{multline*}
 r :=\sup \bigg\{ s \in [0, \delta) : \; \Vert \nabla\psi_2(u(t))\Vert \le \frac 12 \left\|\Ar u(t)\right\|\ \text{and} \\
E(t)> \frac{1}{4} \|u_t(t)\|^2 + \frac{1}{4} \|\Ar u(t)\|^2 +\frac{1}{4}\int^{t}_{t-\tau}
 \Vert \nabla \psi_2(u(s))\Vert^2 ds + \frac{1}{4}\int_{t-\tau}^{t}\vert \alpha(s+\tau)\vert \Vert B^*u_t(s) \Vert^2ds \\
 \text{hold for all} \ t \in [0, s] .\bigg\}
\end{multline*}
From \eqref{E1} and \eqref{E2}, the number $r$ is well-defined. We have to show that $r=\delta.$ 
We suppose by contradiction that $r < \delta$. Then, we have to distinguish two cases.
\\{\bf Case i)} Suppose that 
$$\Vert \nabla\psi_2(u(t))\Vert \le\frac 12 \left\|\Ar u(t)\right\|, \forall\ t\in [0,r),\hspace{8,5 cm}$$ 
$$
E(t) > \frac{1}{4} \|u_t(t)\|^2 + \frac{1}{4} \|\Ar u(t)\|^2 +\frac{1}{4}\int^{t}_{t-\tau} 
 \Vert \nabla \psi_2(u(s))\Vert^2 ds + \frac{1}{4} \int_{t - \tau}^t |\alpha(s+\tau)| \cdot \|B^* u_t(s)\|^2 ds ,\ \forall\ t\in [0,r),
$$
and
$$ E(r) = \frac{1}{4} \|u_t(r)\|^2 + \frac{1}{4} \|\Ar u(r)\|^2+\frac{1}{4}\int^{r}_{r-\tau}
 \Vert \nabla \psi_2(u(s))\Vert^2 ds + \frac{1}{4} \int_{r - \tau}^r |\alpha(s+\tau)| \cdot \|B^* u_t(s)\|^2 ds.
$$

In particular, the last identity implies that
\[
\frac{1}{4} \|u_t(r)\|^2+\frac{1}{4} \Vert \Ar u(r)\Vert^2 \leq E(r).
\]
Thus, the estimates \eqref{aggiunta} and \eqref{cdt_energy} are satisfied on $[0,r).$ Then, from Proposition \ref{pro_1} and the monotonicity of $h_1$, we get
\begin{equation}\label{cdt_r_2}
 h_1\left(\Vert \Ar u(r)\Vert\right) \leq h_1\left(2{E^{1/2}(r)} \right)\leq h_1\left(2\sqrt{2}C^{1/2}(T)E^{1/2}(0)\right) <\frac{1 }{2}.
\end{equation}
Finally, using \eqref{cdt_psi_1} and \eqref{cdt_r_2}, we can conclude that
\begin{eqnarray*}
E(r) &=& \frac{1}{2} \| u_t(r) \|^2 + \frac{1}{2} \| \Ar u(r) \|^2 - \psi_1(u(r))+\frac{1}{2}\int^{r}_{r-\tau}
 \Vert \nabla \psi_2(u(s))\Vert^2 ds \\
& &+ \frac{1}{2}\int_{r-\tau}^{r}\vert \alpha(s+\tau)\vert \Vert B^*u_t(s) \Vert^2ds \\
&>& \frac{1}{4} \| u_t(r) \|^2 + \frac{1}{4} \| \Ar u(r) \|^2 +\frac{1}{4}\int^{r}_{r-\tau}
 \Vert \nabla \psi_2(u(s))\Vert^2 ds + \frac{1}{4} \int_{r - \tau}^r |\alpha(s+\tau)| \| B^* u_t(s) \|^2 ds.
\end{eqnarray*}
contradicting the maximality of $r$. This implies $r = \delta$.
\\{\bf Case ii)} Suppose that 
$$\Vert \nabla\psi_2(u(t))\Vert \le\frac 12 \left\|\Ar u(t)\right\|, \forall\ t\in [0,r],\hspace{8,5 cm}$$
$$
E(t) > \frac{1}{4} \|u_t(t)\|^2 + \frac{1}{4} \|\Ar u(t)\|^2 +\frac{1}{4}\int^{t}_{t-\tau} 
 \Vert \nabla \psi_2(u(s))\Vert^2 ds + \frac{1}{4} \int_{t - \tau}^t |\alpha(s+\tau)| \cdot \|B^* u_t(s)\|^2 ds ,\ \forall\ t\in [0,r],
$$
and
\begin{equation}\label{E3}
\Vert \nabla\psi_2(u(s))\Vert >\frac 12 \left\|\Ar u(s)\right\|, \ s\in (r, r+\epsilon),
\end{equation}
for some $\epsilon >0$ with $ r+\epsilon<\delta.$
However, by continuity, there exists $\epsilon'>0$ (take $\epsilon'<\epsilon$) such that the lower bound on the energy still holds on $[r, r+\epsilon'],$
namely,
$$
E(t) > \frac{1}{4} \|u_t(t)\|^2 + \frac{1}{4} \|\Ar u(t)\|^2 +\frac{1}{4}\int^{t}_{t-\tau} 
 \Vert \nabla \psi_2(u(s))\Vert^2 ds + \frac{1}{4} \int_{t - \tau}^t |\alpha(s+\tau)| \cdot \|B^* u_t(s)\|^2 ds ,
$$
for all $t\in [0, r+\epsilon').$
Moreover, for $\epsilon'$ small enough, we have
$$
\Vert \nabla\psi_2(u(s))\Vert \le \|\Ar u(s)\|, \ s\in (r, r+\epsilon').
$$
Therefore, \eqref{aggiunta} and \eqref{cdt_energy} are satisfied and we can apply Proposition \ref{pro_1} on the larger interval $[0, r+\epsilon').$ We then obtain
$$
 \Vert \Ar u(r+\epsilon')\Vert\leq 2{E^{1/2}(r+\epsilon')} \leq 2\sqrt{2}C^{1/2}(T)E^{1/2}(0),$$
that together with (A2)-(A3) implies
$$\Vert \nabla\psi_2(u(r+\epsilon'))\Vert \le L\left (2\sqrt{2}C^{1/2}(T)E^{1/2}(0)\right )\left\|\Ar u(r+\epsilon')\right\|<\frac 12\left\|\Ar u(r+\epsilon')\right\|,$$
in contradiction with \eqref{E3}. Then, it must be $r=\delta.$
\end{proof}

 We are now in a position to establish that, for sufficiently small initial data, the solutions to system \eqref{p} are always defined and satisfy the well-posedness assumptions \eqref{U_rho_th}-\eqref{serve} of Theorem \ref{th_stability}.

\begin{theorem}\label{Stability}
Assume \eqref{cdt_alpha_2} and \eqref{cdt_alpha}. Then, there exist $\rho>0$ and $C_\rho>0,$ with
$L(C_\rho)<\frac {\omega-\omega'}{M(1+e^{\omega\tau})}$, for which if $\Phi=(f_0,g_0)\in C([-\tau,0]; \He)$ satisfies
\begin{equation}\label{smallnessID}
	\|u_1\|^2 + \|\Ar u_0\|^2 +\int^{0}_{-\tau}
	\Vert \nabla \psi_2(u_0(s))\Vert^2 ds +\int_{-\tau}^0 |\alpha(s+\tau)| \cdot \|B^* g_0(s)\|^2 \, ds < \rho^2,
\end{equation}
and
\begin{equation}\label{additional}
	\max _{s\in [-\tau,0]} \Vert \Phi(s)\Vert_{\mathcal H}<\rho,
	\end{equation}
then, problem \eqref{P_abstract} has a unique global solution 
satisfying 
$$\Vert U(t)\Vert_{\mathcal H}\le C_\rho,\quad \forall \ t\ge 0.$$
\end{theorem}
\begin{proof}
Let $T$ be a fixed time with $T \geq \tau$ such that,
	\begin{equation*}
C_T :=
 2M^2 e^{2\gamma}\max \left\{\left(1+ e^{\tau} (b\sigma+\tau)\right), e^{\omega\tau}\right \}\left(1 + e^{2\omega \tau} (\sigma+\tau)^2 \right) e^{-(\omega - \omega') T}
 < 1. 
\end{equation*}
Also, we denote
\begin{equation*}
C^*_T :=\sup \left\{e^{2T+2b\int_{nT}^{(n+1)T} \left(\vert\alpha(s)\vert+\vert\alpha(s+\tau)\vert\right) ds},\; n\in \N \right\}.
\end{equation*}
Condition \eqref{cdt_alpha} guarantees that $C_T^*$ is finite. Note that $C_T^* \geq C(T)$, where $C(T)$ is defined in \eqref{constant_t}. 
Next, let $\rho > 0$ be chosen such that
\begin{equation}\label{rho_condition}
\rho \leq \frac{1}{2 \sqrt{2C^*_T}} h^{-1}\left(\frac{1}{2}\right)\quad \mbox{and}\quad L\left(2 \sqrt{2C^*_T}\rho\right)<\frac12,
\end{equation}
and assume the conditions \eqref{smallnessID} and \eqref{additional} are verified, i.e.,
\begin{equation*}
\|u_1\|^2 + \|\Ar u_0\|^2 +\int^{0}_{-\tau}
 \Vert \nabla \psi_2(u_0(s))\Vert^2 ds +\int_{-\tau}^0 |\alpha(s+\tau)| \cdot \|B^* g_0(s)\|^2 \, ds < \rho^2, 
\end{equation*}
and
$$\max _{s\in [-\tau,0]} \Vert \Phi(s)\Vert_{\mathcal H}<\rho.$$
This implies (by considering the abstract setting, the system \eqref{P_abstract}) that
\begin{equation}\label{rho_abstract}
\|U_0\|^2 +\int_{0}^{\tau}
 \Vert \nabla \psi_2(u_0(s-\tau))\Vert^2 ds + \int_0^\tau |\alpha(s)| \cdot \|B^*g_0(s-\tau)\|^2\, ds < \rho^2.
\end{equation}

From Lemma \ref{lem_abst_exis}, we know that there exists a solution \( u \) to problem \eqref{p} on a time interval \( [0,\delta) \). Now, we have
\[
h_1\left(\left\|\Ar u_0\right\| \right) \leq h_1(\rho) \leq h_1\left( \frac{1}{2\sqrt{2C^*_T}} h_1^{-1}\left(\frac{1}{2}\right) \right) < \frac{1}{2},
\]

using the fact that \( C^*_T > 1 \). Thus, by Lemma \ref{exis_lemma}, \( E(0) > 0 \). Moreover, from \eqref{cdt_psi_1}, we obtain
\[
E(0) \leq \frac{1}{2} \|u_1\|^2 + \frac{3}{4} \|\Ar u_0\|^2 +\frac{1}{2}\int^{0}_{-\tau}
 \Vert \nabla \psi_2(u_0(s))\Vert^2 ds + \frac{1}{2} \int_{-\tau}^0 |\alpha(s+\tau)| \cdot \|B^* g_0(s)\|^2 \, ds \leq \rho^2,
\]
which gives us
\begin{equation}\label{h_cdt_1_2}
 h_1\left(2 \sqrt{2}\sqrt{C^*_T E(0)} \right) < h_1\left( 2\sqrt{2}\sqrt{C^*_T} \rho \right) < h_1\left( h_1^{-1}\left( \frac{1}{2} \right) \right) = \frac{1}{2},
\end{equation}
and
\begin{equation}\label{suL}
L\left(2 \sqrt{2}\sqrt{C^*_T E(0)}\right) \leq L\left( 2\sqrt{2}\sqrt{C^*_T} \rho \right) < \frac{1}{2}.
\end{equation}
Therefore, by applying Lemma \ref{exis_lemma} once more, we can conclude that \eqref{estimate_psi2} and \eqref{U_cdt_E} hold for all \( t \in [0, \delta) \). Then, using Proposition \ref{pro_1}, we obtain
\begin{multline}\label{eq_E_bound}
\qquad 0< \frac{1}{4} \| u_t(t) \|^2 + \frac{1}{4} \| \Ar u(t) \|^2 +\frac{1}{4}\int^{t}_{t-\tau}
 \Vert \nabla \psi_2(u(s))\Vert^2 ds\\ + \frac{1}{4} \int_{t - \tau}^t |\alpha(s+\tau)| \| B^* u_t(s) \|^2 ds \leq E(t)\leq C^*_T E(0), \qquad \qquad
\end{multline}
for any \( t \in [0, \delta) \). Then, we can extend the solution on the whole interval \( [0, T] \) and \eqref{eq_E_bound} continues to hold on $[0,T].$
 From \eqref{h_cdt_1_2} and \eqref{suL}, for \( t = T \), we have
\begin{equation}\label{sqrt_A_u_t}
h_1\left( \| \Ar u(T)\| \right) \leq h_1(2\sqrt{E(T)}) \leq h_1\left( 2\sqrt{C^*_T E(0)} \right) \leq h_1\left(2 \sqrt{C^*_T} \rho \right) < \frac{1}{2}.
\end{equation}

Moreover, from the assumption \eqref{rho_abstract} on the initial datum $\Phi$ and
from \eqref{eq_E_bound}, we arrive at
\[
\frac{1}{4} \|U(t)\|_{\He}^2 \leq E(t) \leq C^*_T E(0) < C^*_T \rho^2, \quad \forall t \in [0, T],
\]
and so
\begin{equation}\label{U_rdo_leq}
\|U(t)\| \leq C_{\rho} := 2\sqrt{C^*_T} \rho,
\end{equation}
for any \( t \in [0, T] \). By choosing smaller values of \( \rho \), we assume that \( \rho \) is such that \( L(C_\rho) < \frac{\omega - \omega'}{2M(1+e^\tau)} \). Thus, the well-posedness assumption \eqref{U_rho_th} is satisfied on the interval \( [0, T] \). Applying Theorem \ref{th_stability}, we obtain the following estimate:
\begin{equation}\label{eq_23}
\|U(t)\| \leq M e^{\gamma} \left( \|U_0\| +\beta + \int_0^\tau e^{\omega s} |\alpha(s)| \cdot \|\Phi(s - \tau)\|_{\He} ds \right) e^{-\frac{\omega - \omega'}{2} t}, 
\end{equation}
for any $t \in [0, T]$. By using assumption \eqref{cdt_alpha}, \eqref{additional}, and H\"{o}lder inequality, we get
\[
\int_0^{\tau} |\alpha(s)| e^{\omega s} \|\Phi(s - \tau)\|_{\He} ds \leq e^{\omega \tau} \sqrt{\int_0^{\tau} |\alpha(s)| ds }\sqrt{\int_0^{\tau} |\alpha(s)| \cdot \|\Phi(s - \tau)\|_{\He}^2 ds }\leq e^{\omega \tau}{\sigma} \rho.
\]
Analogously, by \eqref{beta_eq}, \eqref{additional}, and \eqref{suL}, we obtain
$$\beta\le \frac 12 \int_0^\tau e^{\omega s} \Vert \Phi (s-\tau)\Vert_{\mathcal{H}} ds \le \frac 12 e^{\omega\tau}\rho \tau.
$$

From \eqref{eq_23} we thus get
\begin{eqnarray*}\label{eq_24}\notag
 \|U(t)\|^2_{\He} &\leq& M^2 e^{2\gamma} \left( \|U_0\| + e^{\omega \tau} \rho (\sigma+\tau) \right)^2 e^{-(\omega - \omega')t}, \\
 &\leq&2 M^2 e^{2\gamma} \left( \|U_0\|^2 + e^{2\omega \tau} (\sigma+\tau)^2 \rho^2 \right) e^{-(\omega - \omega')t}, 
\end{eqnarray*}
which, by using \eqref{rho_abstract}, leads to
\begin{equation}\label{U_estimate_bound_above}
 \|U(t)\|^2_{\He} \leq 2M^2 \rho^2 e^{2\gamma} \left(1 + e^{2\omega \tau} (\sigma+\tau)^2 \right) e^{-(\omega - \omega') t}, 
\end{equation}
for all \( t \in [0, T] \). Also, for all \( s \in [T - \tau, T] \subseteq [0, T] \), \eqref{U_estimate_bound_above} yields
\begin{eqnarray*}
 \Vert u_t(s)\Vert^2 &\leq & \|U(s)\|_{\mathcal{H}}^2 \leq 2M^2 \rho^2 e^{2\gamma} \left(1 + e^{2\omega \tau} (\sigma+\tau )^2 \right) e^{-(\omega - \omega') s} \\
 &\leq & 2M^2 \rho^2 e^{2\gamma+\omega\tau} \left(1 + e^{2\omega \tau} (\sigma+\tau)^2 \right) e^{-(\omega - \omega') T}, 
\end{eqnarray*}
Consequently, by \eqref{cdt_alpha},
\begin{equation}\label{nuova1}
\begin{array}{l}
\displaystyle{ \int_{T-\tau}^{T} |\alpha(s)| \cdot \|B^* u_t(s)\|^2 ds \leq 2M^2 \rho^2 e^{2\gamma+\omega\tau} b\left(1 + e^{2\omega \tau} (\sigma+\tau)^2 \right) e^{-(\omega - \omega') T} \int_{T-\tau}^{T} |\alpha(s)| ds} \\
 \displaystyle{ \hspace{4.3 cm}\leq 2M^2 \rho^2 e^{2\gamma+\omega\tau} b\sigma\left(1 + e^{2\omega \tau} (\sigma+\tau)^2 \right) e^{-(\omega - \omega') T}. }
 \end{array}
\end{equation}
Moreover, from
 \begin{equation*}
 \int^{T}_{T-\tau} \Vert \nabla \psi_2(u(s))\Vert^2 ds \leq \int^{T}_{T-\tau} L^2\left(\left\|\sqrt{A}u(s)\right\|\right) \left\|\sqrt{A}u(s)\right\|^2ds,
 \end{equation*}
since we have
$$ L\left( \| \Ar u(s)\| \right) \leq L (2\sqrt{E(s)}) \leq L\left( 2\sqrt{2}\sqrt{C_T^* E(0)} \right) < \frac{1}{2}<1,$$
and 
\begin{eqnarray*}
 \left\|\sqrt{A}u(s)\right\|^2 &\leq & \|U(s)\|_{\mathcal{H}}^2 \leq 2M^2 \rho^2 e^{2\gamma} \left(1 + e^{2\omega \tau} (\sigma+\tau)^2 \right) e^{-(\omega - \omega') s} \\
 &\leq & 2M^2 \rho^2 e^{2\gamma+\omega\tau} \left(1 + e^{2\omega \tau} (\sigma+\tau)^2 \right) e^{-(\omega - \omega') T}, 
\end{eqnarray*}
it results
\begin{equation}\label{cdt_bound_psi}
 \int^{T}_{T-\tau} \Vert \nabla \psi_2(u(s))\Vert^2 ds \leq 2M^2 \rho^2 e^{2\gamma+\omega\tau} \tau\left(1 + e^{2\omega \tau} (\sigma+\tau)^2 \right) e^{-(\omega - \omega') T} .
\end{equation} 
This last fact, together with \eqref{U_estimate_bound_above} and \eqref{nuova1}, implies that
\begin{multline}\label{U_T_rho}
 \|U(T)\|_{\mathcal{H}}^2 + \int^{T}_{T-\tau} \Vert \nabla \psi_2(u(s))\Vert^2 ds + \int_{T-\tau}^{T} |\alpha(s)| \cdot \|B^* u_t(s)\|_{\mathcal{H}}^2 \, ds \\
 \leq 2M^2 \rho^2 e^{2\gamma}\left(1+ e^{\omega\tau} (b\sigma+\tau)\right)\left(1 + e^{2\omega \tau} (\sigma+\tau)^2 \right) e^{-(\omega - \omega') T} \leq C_T \rho^2 < \rho^2.
\end{multline}
Moreover, from \eqref{U_estimate_bound_above}, we have
\begin{equation}\label{U_rho}
\max_{s\in [T-\tau, T]} \|U(s)\|_{\mathcal{H}}^2< 
 2M^2 \rho^2 e^{2\gamma+\omega\tau} \left(1 + e^{2\omega \tau} (\sigma+\tau)^2 \right) e^{-(\omega - \omega') T}\le C_T\rho^2<\rho^2.
\end{equation}

Conditions \eqref{U_T_rho} and \eqref{U_rho} allow us to replicate the previous arguments on the interval \([T, 2T]\). Specifically, we consider the initial value problem
\begin{equation}\label{P_abstract_222}
\left\{
\begin{array}{ll}
W_t(t) = \A W(t) +\alpha(t)\B W(t - \tau)+ F(W(t), W(t - \tau)), &\quad t \in [T, 2T], \\
W(s) = U(s), \quad s \in [T - \tau, T],
\end{array}
\right.
\end{equation}
where \(U(\cdot)\) denotes the solution to \eqref{p} on \([0, T]\).
Now, let us introduce the energy functional for the solution:
\begin{eqnarray}\label{function_E_222} \notag
 \mathcal{E}(t)&=& \frac{1}{2}\left\|w_{t}(t)\right\|^{2}+\frac{1}{2}\left\|\sqrt{A}w(t)\right\|^{2}-\psi_1(u(t))+\frac{1}{2}\int^{t}_{t-\tau}
 \Vert \nabla \psi_2(w(s))\Vert^2 ds \\
& &+ \frac{1}{2}\int_{t-\tau}^{t}\vert \alpha(s+\tau)\vert \Vert B^*w_t(s) \Vert^2ds . 
\end{eqnarray}
Observe that \(\mathcal{E}(T) = E(T)\). By Lemma \ref{lem_abst_exis}, problem \eqref{P_abstract_222} with initial data \(U(s)\), \(s \in [T - \tau, T]\), admits a unique local solution \(W(\cdot)\) on \([T, T + \delta)\) given by Duhamel's formula:
\begin{equation}\label{Formule_W_2T}
W(t) = S(t - T) U(T) + \int_T^t S(t - T - s) \bigl[\alpha(s)\B W(s - \tau)+ F(W(s), W(s - \tau))\bigr] \, ds,
\end{equation}
for all \(t \in [T, T + \delta)\). We may assume \(\delta \leq \tau\) and that \(W\) is nontrivial; otherwise, \eqref{U_rho} holds trivially for all \(t \geq T\).
From \eqref{sqrt_A_u_t}, we have \(\mathcal{E}(T) > 0\). If
 $$ \mathcal{E}(t) \geq \frac{1}{4} \|w_t(t)\|^2+\frac{1}{4} \|\sqrt{A} w(t)\|^2,$$
 and 
$\Vert \nabla\psi_2(w(t))\Vert\le \Vert \sqrt{A}w(t) \Vert$, for all \(t \in [T, T + \delta)\), then, following Proposition \ref{pro_1}, we obtain the growth estimate
\begin{equation}\label{estiamte_energy_22}
\mathcal {E}(t)\leq e^{2\int_{T}^{t}[ 1+ \left(\vert\alpha(s)\vert+\vert\alpha(s+\tau)\vert\right)b] ds}\mathcal{E}(T), \quad \forall t \in [T, T + \delta) ,
\end{equation}
Additionally, by applying the same reasoning as in Lemma \ref{exis_lemma} and noting that $$C^*_T \geq e^{2\int_{T}^{2T}[ 1+ \left(\vert\alpha(s)\vert+\vert\alpha(s+\tau)\vert\right)b] ds},$$ if
\begin{equation}\label{cdt_h_T_2T}
 h_1\left(2\sqrt{2}\sqrt{C^*_T\mathcal{E}(T)}\right) <\frac{1 }{2} \quad \mbox{and} \quad L \left(2\sqrt{2}\sqrt{C^*_T\mathcal{E}(T)}\right) <\frac{1 }{2},
\end{equation}
then, for every $t \in [T, T + \delta)$,
\begin{eqnarray}\label{function_E_333} \notag
 \mathcal{E}(t)&>& \frac{1}{4}\left\|w_{t}(t)\right\|^{2}+\frac{1}{4}\left\|\sqrt{A}w(t)\right\|^{2}+\frac{1}{4}\int^{t}_{t-\tau}
 \Vert \nabla \psi_2(w(s))\Vert^2 ds \\
& &+ \frac{1}{4}\int_{t-\tau}^{t}\vert \alpha(s+\tau)\vert \Vert B^*w_t(s) \Vert^2ds ,
\end{eqnarray}
and 
	\begin{equation}\label{quasi}
	\Vert \nabla\psi_2(w(t))\Vert\le \Vert \sqrt{A}w(t) \Vert.
	\end{equation}
This implies, in particular,
\begin{equation}\label{E_W_cdt}
 \mathcal{E}(t)>\frac{1}{4}\left\|W(t)\right\|_\He^{2}, \quad \forall t \in [T, T + \delta).
\end{equation} 
Now, observe that, by \eqref{U_T_rho}, we have $\mathcal{E}(T)<\rho$, thus, condition \eqref{cdt_h_T_2T} is satisfied. Consequently, \eqref{function_E_333}, \eqref{quasi}, and \eqref{E_W_cdt} hold for all $t \in [T, T + \delta).$
As a result, the inequality \eqref{estiamte_energy_22} is fulfilled. Combining \eqref{estiamte_energy_22} and \eqref{function_E_333}, we finally get
\begin{multline}\label{eq_E_bound_T}
\frac{1}{4} \|w_t(t)\|^2 + \frac{1}{4} \|\Ar w(t)\|^2 +\frac{1}{4}\int^{t}_{t-\tau}
 \Vert \nabla \psi_2(w(s))\Vert^2 ds \\
 + \frac{1}{4} \int_{t - \tau}^t |\alpha(s+\tau)| \cdot \|B^* w_t(s)\|^2 ds< \mathcal{E}(t) \leq e^{2\int_{T}^{2T}[ 1+ \left(\vert\alpha(s)\vert+\vert\alpha(s+\tau)\vert\right)b] ds} \mathcal{E}(T) \leq C_T^* \mathcal{E}(T).
\end{multline}

Therefore, the solution $W$ remains bounded, allowing us to extend it to the whole interval $[T, 2T]$. Furthermore, the estimates \eqref{eq_E_bound_T} and \eqref{E_W_cdt} remain valid for every $t$ in $[T, 2T]$. Consequently, we obtain
\begin{equation}\label{W_rho_leq}
\|W(t)\|_\He \leq 2\sqrt{C_T^* \mathcal{E}(T)} \leq 2\sqrt{C_T^*} \rho = C_\rho.
\end{equation}
By combining the two partial solutions \eqref{formule_U_rho} and \eqref{Formule_W_2T} to \eqref{P_abstract} obtained on the time intervals $[0, T]$ and $[T, 2T]$, respectively, we establish the existence of a unique solution $U \in C([0, 2T]; \He)$ to \eqref{P_abstract} defined on $[0, 2T]$ and satisfying Duhamel's formula \eqref{formule_U_rho} for all $t \in [0, 2T]$.

Additionally, from \eqref{U_rdo_leq} and \eqref{W_rho_leq}, the solution $U$ satisfies \eqref{U_rho_th} on the time interval $[0, 2T]$ with $C'=C_\rho.$ Hence, the exponential decay estimate \eqref{stab_estimate_2} holds for $U$; which is, 
 \begin{equation}\label{stab_estimate_2_needed}
 \left\Vert U(t)\right\Vert_{\He} \leq Me^\gamma\left(\left\Vert U_0\right\Vert_{\He}+\beta+\int_{0}^{\tau} e^{\omega s}\vert \alpha(s)\vert\cdot\Vert \Phi(s-\tau)\Vert_{\He} ds\right) e^{-\frac{\omega - \omega'}{2} t}.
\end{equation}
for all $t \in [0, 2T]$. Again, we deduce that \eqref{U_estimate_bound_above} holds for all $t \in [0, 2T]$. Furthermore, for all $s \in [2T - \tau, 2T]$, we can obtain estimates 
\eqref{nuova1} and \eqref{cdt_bound_psi}.
Then, similarly to \eqref{U_T_rho}, we get
\begin{multline}\label{U_2T_rho}
 \|U(2T)\|_{\mathcal{H}}^2 + \int^{2T}_{2T-\tau} \Vert \nabla \psi_2(u(s))\Vert^2 ds + \int_{2T-\tau}^{2T} |\alpha(s)| \cdot \|B^* u_t(s)\|_{\mathcal{H}}^2 \, ds \\
 \leq 2M^2 \rho^2 e^{2\gamma}\left(1+e^{\omega\tau} (b\sigma+\tau)\right)\left(1 + e^{2\omega \tau} (\sigma +\tau)^2\right) e^{-(\omega - \omega')2 T} \leq C_T \rho^2 < \rho^2.
\end{multline}
Moreover,
\begin{equation}\label{U_rho_bis}
	\max_{s\in [2T-\tau, 2T]} \|U(s)\|_{\mathcal{H}}^2< 
	2M^2 \rho^2 e^{2\gamma+\omega\tau} \left(1 + e^{2\omega \tau} (\sigma+\tau)^2 \right) e^{-(\omega - \omega') 2T}\le C_T\rho^2<\rho^2.
\end{equation}

Estimates \eqref{U_2T_rho} and \eqref{U_rho_bis} enable us to repeat the same reasoning used earlier on the interval $[2T, 3T]$. This yields the existence of a unique solution to \eqref{P_abstract} defined on $[0, 3T]$ that satisfies \eqref{formule_U_rho} for every $t \in [0, 3T]$. By iterating this process, we ultimately obtain a unique global solution $U \in C([0, +\infty); \He)$ to \eqref{P_abstract}, which fulfills \eqref{formule_U_rho}. 

Consequently, we have shown that, for sufficiently small initial data, solutions to problem \eqref{P_abstract} exist globally in time and are bounded by a positive constant $C_\rho$ satisfying $L(C_\rho)< \frac {\omega-\omega'}{M(1+e^{\omega\tau})}.$ 
\end{proof}

From Theorem \ref{Stability} and Theorem \ref{th_stability}, we immediately deduce the exponential decay estimate for the energy of solutions to \eqref{p} corresponding to {\em small} initial data.

\begin{theorem}\label{last}
	Assume \eqref{cdt_alpha_2} and \eqref{cdt_alpha}. There exists $\rho>0$ such that, if the initial data $\Phi=(f_0,g_0)\in C([-\tau,0]; \mathcal H)$ satisfy
	$$
		\|u_1\|^2 + \|\Ar u_0\|^2 +\int^{0}_{-\tau}
		\Vert \nabla \psi_2(u_0(s))\Vert^2 ds +\int_{-\tau}^0 |\alpha(s+\tau)| \cdot \|B^* g_0(s)\|^2 \, ds < \rho^2,
	$$
	and
	$$
		\max _{s\in [-\tau,0]} \Vert \Phi(s)\Vert_{\mathcal H}<\rho,
	$$
	then
	problem \eqref{p} has a unique global solution 
	satisfying 
	$$E(t)\le \tilde C e^{-(\omega-\omega') t},\quad \forall \ t\ge 0,$$
where $\tilde C$ is a suitable constant depending on the initial data.
\end{theorem}
	\section{Viscoelastic models: system \eqref{visco_p}} \label{visco_dam_sec}
We now turn our attention to the second class of systems considered in this paper, namely the viscoelastic model \eqref{visco_p}, in which the damping is provided by a memory term rather than by a linear velocity feedback. To handle the infinite memory, we introduce the auxiliary variable $\eta^t$, as in Dafermos \cite{dafermos_1970}, by
\begin{equation}\label{def_eta}
\eta^t(s) := u(t) - u(t-s), \qquad s,t \in (0,+\infty),
\end{equation}
so that we can rewrite \eqref{visco_p} in the following way
\begin{equation}\label{visco_P_memory_first_eq}
u_{tt}(t)+ (1-\tilde{k})A u(t)
+ \int_{0}^{+\infty} k(s)\, A \eta^t(s)ds+\alpha(t) BB^* u_t(t-\tau)=\nabla \psi_1(u) + \nabla \psi_2(u(t-\tau)), 
\end{equation}
for all $t \in(0,+\infty)$, with
 \begin{equation}\label{visco_p_memory}
\left\{
\begin{array}{ll}
\eta_t^{t}(s)= -\eta^t_{s}(s) + u_t(t), & \quad t,s \in (0,+\infty), \\
u(t-\tau)=f_{0}(t-\tau), \; u_{t}(t-\tau)=g_{0}(t-\tau)& \quad t \in [0,\tau],\\
u(0)=u_{0}, \quad u_{t}(0)=u_{1}, & \\
	\eta^0(s) = u_0 - f_0(-s), &\quad s \in (0,+\infty),
\end{array}
\right.
\end{equation}

Let us reformulate {\eqref{visco_P_memory_first_eq}-\eqref{visco_p_memory}} 
into the abstract form 
\begin{equation}\label{visco_P_abstract}
	\left\{
	\begin{array}{ll}
		U_t(t) = \hat{\A} U(t) +\alpha(t)\hat{\B} U(t - \tau)+ \hat{F}(U(t), U(t - \tau)), &\quad t \in (0, +\infty), \\
		U(t - \tau)=\hat{\Phi}(t-\tau), &\quad t \in [0, \tau],
	\end{array}
	\right.
\end{equation}
where $U(t) = (u(t),v(t),\eta^t(t) )^T$ with $v(t) = u_t(t)$, and $\hat{\Phi}(t-\tau)=(f_0(t-\tau), g_0(t-\tau), \eta^0).$ We introduce the Hilbert space $$\hat{\He}=D(\sqrt{A})\times H \times L_k^{2}\bigl((0,+\infty); D(\sqrt{A})\bigr),$$
where $L_k^{2}\bigl((0,+\infty); D(\sqrt{A})\bigr)$ be the Hilbert space of \(D(\sqrt{A})\)-valued functions on \((0,+\infty)\),
endowed with the scalar product
\[
\langle \varphi_1 , \varphi_2\rangle_{L_k^{2}((0,+\infty);D(\sqrt{A}))}
= \int_{0}^{\infty} k(s)\,
\langle \sqrt{A}\varphi_1 , \sqrt{A}\varphi_2 \rangle_H \, ds.
\]
The Hilbert space $\hat{\He}$ is equipped with the inner product
$$\left\langle \begin{pmatrix} u \\ v \\ w \end{pmatrix},
\begin{pmatrix} \tilde{u} \\ \tilde{v} \\ \tilde{w} \end{pmatrix}
\right\rangle_{\hat{\He}}
= (1-\tilde{k}) \langle \sqrt{A} u, \sqrt{A} \tilde{u} \rangle_H
+ \langle v, \tilde{v} \rangle_H
+ \int_{0}^{\infty} k(s)\,
\langle \sqrt{A} w, \sqrt{A} \tilde{w} \rangle_H\, ds.$$
The linear operators $\hat{\A}$ with its domain are given by
$$\hat{\A} \begin{pmatrix} u \\ v \\w \end{pmatrix} = \begin{pmatrix} v \\ -(1-\tilde{k})Au - \int_{0}^{\infty} k(s) A w(s) ds \\ - w_s+ v \end{pmatrix}, $$ 
\begin{equation}\label{visco_domain_A}
D(\hat{\A})
= \left\{ 
\begin{array}{l}
(u,v,w) \in D(\sqrt{A}) \times D(\sqrt{A})
\times L_k^{2}\bigl((0,+\infty); D(\sqrt{A})\bigr) : \\
(1-\tilde{k})u + \int_{0}^{\infty} k(s) w(s) ds \in D(A),
\ w_s \in L_k^{2}\bigl((0,+\infty); D(\sqrt{A})\bigr)
\end{array} \right\}. 
\end{equation}
The operator $\hat{\B}$ and the function $\hat{F}$ are defined as follows
$$ \hat{\B} \begin{pmatrix} u \\ v \\ w \end{pmatrix} = \begin{pmatrix} 0 \\ -BB^*v \\ 0 \end{pmatrix} \quad
\text{and} 
\quad \hat{F}(U(t), U(t - \tau)) = \begin{pmatrix}
0 \\
\nabla \psi_1(u(t)) + \nabla \psi_2(u(t - \tau))\\
0
\end{pmatrix}.$$

It is well-known that
 the operator $\hat{\A}$ in \eqref{visco_P_abstract}, corresponding to the linear nondelayed part of the system, generates an exponentially stable semigroup $\{\hat{S}(t)\}_{ t\geq 0}$ on $\hat{\He} $, see \cite{Giorgi}.
 Then, for positive constants $\hat{M},\hat{\omega},$ it holds
\begin{equation}\label{visco_decay_semigroup}
 \Vert \hat{S}(t)\Vert \le \hat{M} e^{-\hat{\omega}t}.
 \end{equation}
Moreover, the previous assumptions (A2)-(A3) on $\psi_i,$ $i=1,2,$ imply that $\hat{F}$ satisfies (F1) and (F2) above.

The reformulation \eqref{visco_P_abstract} allows us to apply a semigroup strategy analogous to the one developed in the previous section. We first prove a local existence result, then extend it and prove exponential decay under suitable smallness assumptions on the initial data.
\begin{lemma}\label{visco_lem_abst_exis}
 Let us consider the system \eqref{visco_P_abstract} with initial data $\hat{\Phi} \in C([-\tau,0]; \hat{\He})$. Then, there exists a unique continuous local solution $U$ defined on
a time interval $[0, T)$, with $T>0,$ given by 
\begin{equation}\label{visco_formule_U_rho}
U(t) = \hat{S}(t) U_0 + \int_0^t \hat{S}(t -s) \bigl[\alpha(s)\hat{\B} U(s - \tau)+ \hat{F}(U(s), U(s - \tau))\bigr] \, ds,
\end{equation}
for all $t \in [0, T)$.
\end{lemma}
\begin{proof}
The system \eqref{visco_P_abstract} can be rewritten as a nondelayed inhomogeneous Cauchy problem on the interval $[0,\tau]$. Applying the classical theory of nonlinear semigroups (as in the proof of Lemma \ref{lem_abst_exis}) yields the existence of a unique local solution on $[0,\delta)$, with $\delta \leq \tau$, given by Duhamel's formula \eqref{visco_formule_U_rho}.
\end{proof}

\subsection{Exponential decay and well-posedness of model \eqref{visco_p}}
In the viscoelastic setting, we require the following assumption on the linear delay feedback, which is the analogue of \eqref{cdt_alpha_2}:

For suitable constants $\gamma\geq 0$ and $\omega^\prime\in [0,\hat{\omega})$, it holds
\begin{equation}\label{visco_cdt_alpha_2}
	\hat{M} b e^{\hat{\omega} \tau}\int_{0}^{t}\vert \alpha(s+\tau)\vert ds\leq \gamma+\omega^\prime t,\quad \forall t>0, 
\end{equation}
where $b=\Vert\hat{\B}\Vert_{\hat{\mathcal{H}}}=\Vert BB^*\Vert$ and $\hat{M}, \hat{\omega}$ are the parameters in \eqref{visco_decay_semigroup}.

The following theorem provides the exponential decay estimate for the viscoelastic system. Its proof follows the same arguments as Theorem \ref{th_stability}, adapted to the augmented state space $\hat{\He}.$
\begin{theorem}\label{visco_th_stability}
	Assume \eqref{visco_cdt_alpha_2}.
	Let $U$ be a solution to \eqref{visco_P_abstract}, defined on a time interval $[0, T'),$ 
	satisfying
	\begin{equation}\label{visco_U_rho_th}
		\Vert U(t)\Vert \leq C' \quad \forall t\in [-\tau, T'),
	\end{equation}
	for some $C'>0$ such that 
\begin{equation}\label{serve2}
L(C')<\frac {\hat{\omega}-\omega'}{\hat{M}(1+e^{\hat{\omega}\tau})}.
\end{equation}
		Then, $U$ decays exponentially: 
	\begin{equation}\label{visco_stab_estimate_2}
		\left\Vert U(t)\right\Vert_{\hat{\He}} \leq \hat{M}e^{\gamma}\left(\left\Vert U_0\right\Vert_{\hat{\He}}+\hat{\beta}+\int_{0}^{\tau} e^{\hat{\omega} s}\vert \alpha(s)\vert\cdot\Vert \hat{\Phi}(s-\tau)\Vert_{\hat{\He}} ds\right) e^{-(\hat{\omega}-\omega^\prime-\hat{M}L(C')(1+e^{\hat{\omega} \tau}))t},
	\end{equation}
	for all $t\in [0,T'),$ where
	\begin{equation}\label{visco_beta_eq}
	 \hat{\beta} = L(C')\int_0^\tau e^{\hat{\omega} s} \Vert \hat{\Phi}(s-\tau)\Vert_{\hat{\mathcal{H}}}ds. 
	\end{equation}
	\end{theorem}

\begin{remark}\label{visco_remark_1}
	Also here, we refer to the assumption \eqref{visco_U_rho_th} in Theorem \ref{visco_th_stability}, with $C'$ satisfying \eqref{serve2}, as the well-posedness assumption.
\end{remark}

 Therefore, establishing the exponential stability of the solutions to system \eqref{visco_p} requires verifying that the well-posedness condition \eqref{visco_U_rho_th} holds.
 
 For this purpose, we introduce the energy functional related to system \eqref{visco_p} as follows
\begin{eqnarray}\label{visco_function_E} \notag
	\hat{E}(t)&=& \frac{1}{2}\left\|u_{t}(t)\right\|^{2}+\frac{ 1 - \tilde{k}}{2}\left\|\sqrt{A}u(t)\right\|^{2}-\psi_1(u(t))+\frac{1}{2}\int^{t}_{t-\tau}
	\Vert \nabla \psi_2(u(s))\Vert^2 ds \\
	& &+ \frac{1}{2}\int_{t-\tau}^{t}\vert \alpha(s+\tau)\vert \Vert B^*u_t(s) \Vert^2ds + \frac{1}{2} \int_{0}^{+\infty} k(s) \|\sqrt{A} \eta^t(s)\|^2\, ds , 
\end{eqnarray}
To prove our global well-posedness and stability result, we need some preliminary results.
\begin{pro}\label{visco_pro_1}
Let $u$ be a solution to \eqref{visco_p} defined on some interval $[0, T').$ If
	\begin{equation}\label{visco_aggiunta}
\Vert \nabla\psi_2(u(t))\Vert \le (1-\tilde{k}) \Vert\sqrt{A}u(t)\Vert,\ \forall t\in [0, T'),
	\end{equation}and the functional energy verifies\begin{equation}\label{visco_cdt_energy}
		\hat{E}(t)\geq \frac{1}{4}
			\Vert u_t(t)\Vert^2 + \frac{1-\tilde{k}}{4} \left\Vert\sqrt{A}u(t)\right\Vert^2.
	\end{equation}
	Then,
	\begin{equation}\label{visco_estiamte_energy}
	\hat{E}(t)\leq C(t)\hat{E}(0), \quad \forall t\geq 0 ,
	\end{equation}
where $C(t)$
is defined in \eqref{constant_t}.
\end{pro}
\begin{proof}
	By multiplying the equation \eqref{visco_P_memory_first_eq} by $u_t$, we obtain
	\begin{eqnarray*}
		\langle u_{tt}(t), u_t(t) \rangle & =& \langle -(1-\tilde{k})A u(t)
		- \int_{0}^{+\infty} k(s)\, A \eta^t(s)ds-\alpha(t) BB^* u_t(t-\tau)+\nabla \psi_1(u(t)), u_t(t)\rangle \\
		& & + \langle \nabla \psi_2(u(t-\tau)), u_t(t) \rangle .
	\end{eqnarray*}
	Otherwise, we have
	\begin{eqnarray*}
		\hat{E}^\prime(t)&=& \langle u_{tt}(t), u_t(t) \rangle+(1 - \tilde{k})\langle\sqrt{A}u(t), \sqrt{A}u_t(t) \rangle-\langle \nabla\psi_1(u), u_t(t) \rangle+\frac{1}{2}
		\Vert \nabla \psi_2(u(t))\Vert^2\\
		& &-\frac{1}{2}
		\Vert \nabla \psi_2(u(t-\tau))\Vert^2 + \frac{1}{2}\vert \alpha(t+\tau)\vert \Vert B^*u_t(t) \Vert^2 -\frac{1}{2}\vert \alpha(t)\vert \Vert B^*u_t(t-\tau) \Vert^2 \\
		& & + \int_{0}^{+\infty} k(s)
		\left\langle \sqrt{A} \eta^t(s), \sqrt{A}\, \eta^t_{t}(s) \right\rangle ds.
	\end{eqnarray*}
	By combining the previous estimates, we obtain
	\begin{eqnarray*}
		\hat{E}^\prime(t)&=& -\langle \int_{0}^{+\infty} k(s)\, A \eta^t(s)ds, u_t(t)\rangle-\alpha(t) \langle BB^* u_t(t-\tau), u_t(t) \rangle+\langle\nabla \psi_2(u(t-\tau)), u_t(t) \rangle \\
		& &+\frac{1}{2}
		\Vert \nabla \psi_2(u(t))\Vert^2-\frac{1}{2}
		\Vert \nabla \psi_2(u(t-\tau))\Vert^2 + \frac{1}{2}\vert \alpha(t+\tau)\vert \Vert B^*u_t(t) \Vert^2 \\
		& &-\frac{1}{2}\vert \alpha(t)\vert \Vert B^*u_t(t-\tau) \Vert^2 + \int_{0}^{+\infty} k(s)
		\left\langle \sqrt{A} \eta^t(s), \sqrt{A}\, \eta^t_{t}(s) \right\rangle ds, 
	\end{eqnarray*}
	by the first equation in system \eqref{visco_p_memory}, and we find
	\begin{eqnarray*}
	\hat{E}^\prime(t)&=& -\alpha(t) \langle BB^* u_t(t-\tau), u_t(t) \rangle+\langle\nabla \psi_2(u(t-\tau)), u_t(t) \rangle +\frac{1}{2}
		\Vert \nabla \psi_2(u(t))\Vert^2\\
		& &-\frac{1}{2}
		\Vert \nabla \psi_2(u(t-\tau))\Vert^2 + \frac{1}{2}\vert \alpha(t+\tau)\vert \Vert B^*u_t(t) \Vert^2 -\frac{1}{2}\vert \alpha(t)\vert \Vert B^*u_t(t-\tau) \Vert^2 \\
		& & - \int_{0}^{+\infty} k(s)
		\left\langle \sqrt{A} \eta^t(s), \sqrt{A}\, \eta^t_{s}(s) \right\rangle ds.
	\end{eqnarray*}
	Apart from that, by employing that $\eta^t(0)=0$ and $k(s)\|\sqrt{A}\eta^t(s)\|^2\to 0$ as $s\to+\infty$, an integration by parts gives
	\[
	\int_0^{+\infty} k(s)\langle \eta^t_{s}(s),A\eta_t(s)\rangle ds
	= -\frac12\int_0^{+\infty} k^\prime(s)\|\sqrt{A}\eta^t(s)\|^2 ds.
	\]
	Therefore,
	\[
	-\frac12\int_0^{+\infty} k^\prime(s)\|\sqrt{A}\eta^t(s)\|^2 ds
	\ge \frac{\zeta}{2}\int_0^{+\infty} k(s)\|\sqrt{A}\eta^t(s)\|^2 ds \ge 0.
	\] 
	With the use of this last result and Cauchy-Schwartz inequality, we arrive at
	\begin{eqnarray*}
		\hat{E}^\prime(t)&\leq& \frac{1}{2}
		\Vert u_t(t)\Vert^2 + \frac{1}{2}\Vert \nabla \psi_2(u(t))\Vert^2+\frac{1}{2} \left(\vert\alpha(t)\vert+\vert\alpha(t+\tau)\vert\right) \Vert B^*u_t(t) \Vert^2 , 
	\end{eqnarray*}
	By using the assumption \eqref{visco_aggiunta}, we get
	\begin{equation*}
		\hat{E}^\prime(t)\leq \frac{1}{2}
		\Vert u_t(t)\Vert^2 + \frac{ 1 - \tilde{k}}{2}\left\Vert\sqrt{A}u(t)\right\Vert^2+\frac{1}{2} \left(\vert\alpha(t)\vert+\vert\alpha(t+\tau)\vert\right)b \Vert u_t(t) \Vert^2 , 
	\end{equation*}
	thus,
	\begin{eqnarray*}
	\hat{E}^\prime(t)&\leq& \frac{1}{2}
		\Vert u_t(t)\Vert^2 + \frac{ 1 - \tilde{k}}{2} \left\Vert\sqrt{A}u(t)\right\Vert^2+\frac{1}{2} \left(\vert\alpha(t)\vert+\vert\alpha(t+\tau)\vert\right)b \Vert u_t(t) \Vert^2 \\
		& & + \frac{ 1 - \tilde{k}}{2} \left(\vert\alpha(t)\vert+\vert\alpha(t+\tau)\vert\right)b^2\left\Vert\sqrt{A}u(t)\right\Vert^2.
	\end{eqnarray*}
	Since $ \hat{E}$ satisfies \eqref{visco_cdt_energy}, we arrive at
	\begin{equation*}
		\hat{E}^\prime(t)\leq 2 \left[\left(\vert\alpha(t)\vert+\vert\alpha(t+\tau)\vert\right)b +1\right] \hat{E}(t) , 
	\end{equation*}
	By simple integration, we get
	\begin{equation*}
		\hat{E}(t)\leq \hat{E}(0)+ 2 \int_{0}^{t} \left[\left(\vert\alpha(s)\vert+\vert\alpha(s+\tau)\vert\right)b +1\right] \hat{E}(s)ds , 
	\end{equation*} and
	by using Gr\"{o}nwall's inequality, \eqref{visco_estiamte_energy} is obtained with $C(t)$ given in \eqref{constant_t}.
\end{proof}

\begin{lemma} \label{visco_exis_lemma}
	Let \( U(\cdot) \) be a non-zero solution to system \eqref{visco_P_abstract} defined on the interval \( [0, \delta) \) and let \( T>\delta \). Then,
	\begin{enumerate}
		\item If \( h_1\left(\left\|\sqrt{A} u_0 \right\|\right) < \frac{1-\tilde{k}}{2} \), then \( \hat{E}(0) > 0 \).
		
		\item If 
			\( h_1\left(\frac{2\sqrt{2} }{\sqrt{1-\tilde{k}}} C^{1/2}(T)\hat{E}^{1/2}(0)\right) <\frac{1-\tilde{k}}{2} \), and \( L\left(\frac{2\sqrt{2} }{\sqrt{1-\tilde{k}}} C^{1/2}(T)\hat{E}^{1/2}(0)\right) <\frac{1-\tilde{k}}{2} \),
		where \( C(\cdot) \) is defined in \eqref{constant_t}, then the following estimates hold for all $t \in [0, \delta)$:
	\end{enumerate}
		\begin{equation}\label{visco_estimate_psi2}
			\Vert \nabla\psi_2(u(t))\Vert \le \frac{1-\tilde{k}}{2} \left\|\Ar u(t)\right\|, 
		\end{equation}
	\begin{eqnarray}\notag\label{visco_estimate_energy_existance}
		\hat{E}(t)& >& \frac{1}{4} \|u_t(t)\|^2 + \frac{1-\tilde{k}}{4} \left\|\sqrt{A} u(t)\right\|^2 + \frac{1}{4}\int_{t-\tau}^{t}\vert \alpha(s+\tau)\vert \Vert B^*u_t(s) \Vert^2ds \\
		& & +\frac{1}{4}\int^{t}_{t-\tau}
		\Vert \nabla \psi_2(u(s))\Vert^2 ds +\frac{1}{4}\int_{0}^{+\infty} k(s)\|\sqrt{A}\eta^t(s)\|^{2} ds,
	\end{eqnarray}
	for all $t \in [0, \delta)$. Particularly, we arrive at
	\begin{equation}\label{visco_U_cdt_E}
		\hat{E}(t) > \frac{1}{4} \|U(t)\|_{\hat{\He}}^2, 
	\end{equation}
	for all $t \in [0, \delta)$.
\end{lemma}
\begin{proof}
	
	Combining \eqref{cdt_psi_1} and the assumption $h_1(\|\sqrt{A} u_0\|) < \frac{1 -\tilde{k}}{2}$, we deduce
	\begin{eqnarray*}
	\hat{E}(0)&=& \frac{1}{2} \|u_1\|^2 + \frac{1-\tilde{k}}{2} \|\sqrt{A} u_0\|^2 - \psi_1(u_0) +\frac{1}{2}\int^{0}_{-\tau}
		\Vert \nabla \psi_2(f_0(s))\Vert^2 ds \\
		& &+ 
		\frac{1}{2}\int_{0}^{+\infty} k(s)\|\sqrt{A}\eta^0(s)\|^{2} ds+ \frac{1}{2}\int_{-\tau}^{0}\vert \alpha(s+\tau)\vert \cdot \Vert B^*g_0(s) \Vert^2ds 
		\nonumber \\
		&\geq& \frac{1}{2} \|u_1\|^2 + \frac{1-\tilde{k}}{2} \|\sqrt{A} u_0\|^2 
		- \frac{1}{2} h_1(\|\sqrt{A} u_0\|) \|\sqrt{A} u_0\|^2 +\frac{1}{2}\int^{0}_{-\tau}
		\Vert \nabla \psi_2(f_0(s))\Vert^2 ds \\
		& &+ 
		\frac{1}{2}\int_{0}^{+\infty} k(s)\|\sqrt{A}\eta^0(s)\|^{2} ds + \frac{1}{2} \int_{-\tau}^0 |\alpha(s)| \cdot \|B^* g_0(s)\|^2 ds \nonumber \\
		&>& \frac{1}{4} \|u_1\|^2 + \frac{1-\tilde{k}}{4} \|\sqrt{A} u_0\|^2+\frac{1}{4}\int^{0}_{-\tau}
		\Vert \nabla \psi_2(f_0(s))\Vert^2 ds + 
		\frac{1}{4}\int_{0}^{+\infty} k(s)\|\sqrt{A}\eta^0(s)\|^{2} ds\\ \nonumber
		& & + \frac{1}{4} \int_{-\tau}^0 |\alpha(s+\tau)| \cdot \|B^* g_0(s)\|^2 ds.
	\end{eqnarray*}
	So, we have 
	\begin{multline}\label{visco_eq_E_zero}
			\hat{E}(0)> \frac{1}{4} \|u_1\|^2 + \frac{1-\tilde{k}}{4} \|\sqrt{A} u_0\|^2+\frac{1}{4}\int^{0}_{-\tau}
		\Vert \nabla \psi_2(f_0(s))\Vert^2 ds\\ + \frac{1}{4}\int_{0}^{+\infty} k(s)\|\sqrt{A}\eta^0(s)\|^{2} ds
		+ \frac{1}{4} \int_{-\tau}^0 |\alpha(s+\tau)| \cdot \|B^* g_0(s)\|^2 ds>0,
	\end{multline}
	for $u$ being a nonzero solution. Remark also that, from \eqref{visco_eq_E_zero}, we have 
		$$
		\left\|\Ar u_0\right\| <\frac{2 	\hat{E}^{1/2}(0)}{\sqrt{1-\tilde{k}}}<\frac{2\sqrt{2} }{\sqrt{1-\tilde{k}}} C^{1/2}(T)\hat{E}^{1/2}(0),
		$$
		since $C(T)>1.$ Then, due to the assumptions (A2)-(A3) and from the lemma's assumption $$L\left(\frac{2\sqrt{2} }{\sqrt{1-\tilde{k}}} C^{1/2}(T)	\hat{E}^{1/2}(0)\right)<\frac{1-\tilde{k}}{2},$$ it follows
	\begin{equation}\label{visco_E2}
		\Vert \nabla\psi_2(u_0)\Vert \le L (\|\Ar u_0\| ) \left\|\Ar u_0\right\|\le \frac{1-\tilde{k}}{2} \left\|\Ar u_0\right\| .
	\end{equation}
	To establish the second claim, we proceed by contradiction. Denote
	\begin{multline*}
		r :=\sup \bigg\{ s \in [0, \delta) : \ \Vert \nabla\psi_2(u(t))\Vert \le \frac{1-\tilde{k}}{2} \left\|\Ar u(t)\right\|\ \text{and} \\
		\hat{E}(t)> \frac{1}{4} \|u_t(t)\|^2 + \frac{1-\tilde{k}}{4} \|\Ar u(t)\|^2 +\frac{1}{4}\int^{t}_{t-\tau}
		\Vert \nabla \psi_2(u(s))\Vert^2 ds + 
		\frac{1}{4}\int_{0}^{+\infty} k(s)\|\sqrt{A}\eta^t(s)\|^{2} ds\\ + \frac{1}{4}\int_{t-\tau}^{t}\vert \alpha(s+\tau)\vert \cdot \Vert B^*u_t(s) \Vert^2ds \qquad
		\text{hold for all} \ t \in [0, s] .\bigg\}
	\end{multline*}
		The number $r$ is well-defined according to \eqref{visco_eq_E_zero} and \eqref{visco_E2}. We have to show that $r=\delta.$ 
		We suppose by contradiction that $r < \delta$. Then, we have to distinguish two cases.
	\\{\bf Case i)} Suppose that 
		$$\Vert \nabla\psi_2(u(t))\Vert \le \frac{1-\tilde{k}}{2} \left\|\Ar u(t)\right\|, \quad \forall\ t\in [0,r),$$ 
		\begin{eqnarray*}
			\hat{E}(t)&>& \frac{1}{4} \|u_t(t)\|^2 + \frac{1-\tilde{k}}{4} \|\Ar u(t)\|^2 +\frac{1}{4}\int^{t}_{t-\tau}
			\Vert \nabla \psi_2(u(s))\Vert^2 ds \\
			& & + 
			\frac{1}{4}\int_{0}^{+\infty} k(s)\|\sqrt{A}\eta^t(s)\|^{2} ds+ \frac{1}{4}\int_{t-\tau}^{t}\vert \alpha(s+\tau)\vert \cdot \Vert B^*u_t(s) \Vert^2ds ,
		\end{eqnarray*} 
		for all $t\in [0,r)$, and
	\begin{eqnarray*} 
		\hat{E}(r)& = &\frac{1}{4} \| u_t(r) \|^2 + \frac{1-\tilde{k}}{4} \| \sqrt{A} u(r) \|^2 +\frac{1}{4}\int^{r}_{r-\tau}
		\Vert \nabla \psi_2(u(s))\Vert^2 ds\\
		& & + 
		\frac{1}{4}\int_{0}^{+\infty} k(s)\|\sqrt{A}\eta^r(s)\|^{2} ds + \frac{1}{4} \int_{r - \tau}^r |\alpha(s+\tau)| \cdot \| B^* u_t(s) \|^2 ds .
	\end{eqnarray*}	
	In particular, the last identity implies that
	\[
	\frac{1}{4} \|u_t(r)\|^2+\frac{1-\tilde{k}}{4} \Vert \Ar u(r)\Vert^2 \leq \hat{E}(r).
	\]
		Thus, the estimates \eqref{visco_aggiunta} and \eqref{visco_cdt_energy} are satisfied on $[0,r).$ Then, from Proposition \ref{visco_pro_1} and the monotonicity of $h_1$, we get
	\begin{equation}\label{visco_cdt_r_2}
		h_1\left(\Vert \Ar u(r)\Vert\right) \leq h_1\left(\frac{2}{\sqrt{1-\tilde{k}}}{\hat{E}^{1/2}(r)} \right)\leq h_1\left(\frac{2\sqrt{2}}{\sqrt{1-\tilde{k}}}C^{1/2}(T)\hat{E}^{1/2}(0)\right) <\frac{1-\tilde{k} }{2}.
	\end{equation}
	%
	Finally, using \eqref{cdt_psi_1} and \eqref{visco_cdt_r_2}, we can conclude that
	\begin{eqnarray*}
		\hat{E}(r) &=& \frac{1}{2} \| u_t(r) \|^2 + \frac{1-\tilde{k}}{2} \| \sqrt{A} u(r) \|^2 - \psi_1(u(r))+\frac{1}{2}\int^{r}_{r-\tau}
		\Vert \nabla \psi_2(u(s))\Vert^2 ds \\
		& & + 
		\frac{1}{2}\int_{0}^{+\infty} k(s)\|\sqrt{A}\eta^r(s)\|^{2} ds+ \frac{1}{2}\int_{r-\tau}^{r}\vert \alpha(s+\tau)\vert \cdot\Vert B^*u_t(s) \Vert^2ds \\
		&>& \frac{1}{4} \| u_t(r) \|^2 + \frac{1-\tilde{k}}{4} \| \sqrt{A} u(r) \|^2+\frac{1}{4}\int^{r}_{r-\tau}
		\Vert \nabla \psi_2(u(s))\Vert^2 ds\\
		& & + 
		\frac{1}{4}\int_{0}^{+\infty} k(s)\|\sqrt{A}\eta^r(s)\|^{2} ds+ \frac{1}{4} \int_{r - \tau}^r |\alpha(s+\tau)|\cdot \| B^* u_t(s) \|^2 ds.
	\end{eqnarray*}
	contradicting the maximality of $r$. This implies $r = \delta$.

	{\bf Case ii)} Suppose that 
		$$\Vert \nabla\psi_2(u(t))\Vert \le\frac{1-\tilde{k}}{2} \left\|\Ar u(t)\right\|, \quad \forall\ t\in [0,r],$$
		\begin{eqnarray*}
			\hat{E}(t)&>& \frac{1}{4} \|u_t(t)\|^2 + \frac{1-\tilde{k}}{4} \|\Ar u(t)\|^2 +\frac{1}{4}\int^{t}_{t-\tau}
			\Vert \nabla \psi_2(u(s))\Vert^2 ds \\
			& & + 
			\frac{1}{4}\int_{0}^{+\infty} k(s)\|\sqrt{A}\eta^t(s)\|^{2} ds+ \frac{1}{4}\int_{t-\tau}^{t}\vert \alpha(s+\tau)\vert \cdot \Vert B^*u_t(s) \Vert^2ds ,
		\end{eqnarray*} 
		for all $t\in [0,r]$, and
		\begin{equation}\label{visco_E3}
			\Vert \nabla\psi_2(u(s))\Vert >\frac{1-\tilde{k}}{2} \left\|\Ar u(s)\right\|, \quad s\in (r, r+\epsilon),
		\end{equation}
		for some $\epsilon >0$ with $ r+\epsilon<\delta.$
		However, by continuity, there exists $\epsilon'>0$ (take $\epsilon'<\epsilon$) such that the lower bound on the energy still holds on $[r, r+\epsilon'],$
		namely,
		\begin{eqnarray*}
			\hat{E}(t)&>& \frac{1}{4} \|u_t(t)\|^2 + \frac{1-\tilde{k}}{4} \|\Ar u(t)\|^2 +\frac{1}{4}\int^{t}_{t-\tau}
			\Vert \nabla \psi_2(u(s))\Vert^2 ds \\
			& & + 
			\frac{1}{4}\int_{0}^{+\infty} k(s)\|\sqrt{A}\eta^t(s)\|^{2} ds+ \frac{1}{4}\int_{t-\tau}^{t}\vert \alpha(s+\tau)\vert \cdot \Vert B^*u_t(s) \Vert^2ds ,
		\end{eqnarray*}
		for all $t\in [0, r+\epsilon').$
		Moreover, for $\epsilon'$ small enough, we have
		$$
		\Vert \nabla\psi_2(u(s))\Vert \le \|\Ar u(s)\|, \ s\in (r, r+\epsilon'). 
		$$
		Therefore, \eqref{visco_aggiunta} and \eqref{visco_cdt_energy} are satisfied and we can apply Proposition \ref{visco_pro_1} on the larger interval $[0, r+\epsilon')$ obtaining 
		$$
		\Vert \Ar u(r+\epsilon')\Vert\leq \frac{2 \hat{E}^{1/2}((r+\epsilon'))}{\sqrt{1-\tilde{k}}}\leq\frac{2\sqrt{2} }{\sqrt{1-\tilde{k}}} C^{1/2}(T)\hat{E}^{1/2}(0),$$
		that together with (A2)-(A3) implies
		$$\Vert \nabla\psi_2(u(r+\epsilon'))\Vert \le L\left (\frac{2\sqrt{2}}{\sqrt{1-\tilde{k}}}C^{1/2}(T)\hat{E}^{1/2}(0)\right )\left\|\Ar u(r+\epsilon')\right\|<\frac {1-\tilde{k}}{2}\left\|\Ar u(r+\epsilon')\right\|,$$
		in contradiction with \eqref{visco_E3}. Then, it must be $r=\delta.$
\end{proof}

We are now in a position to establish that, for sufficiently small initial data, the solutions to system \eqref{visco_p} are always defined and satisfy the well-posedness assumptions \eqref{visco_U_rho_th}-\eqref{serve2} of Theorem \ref{visco_th_stability}.
\begin{theorem}\label{visco_Stability}	
	Assume \eqref{visco_cdt_alpha_2} and \eqref{cdt_alpha}. Then, there exist $\rho>0$ and $C_\rho>0,$ with
	$L(C_\rho)<\frac {\hat{\omega}-\omega'}{\hat{M}(1+e^{\hat{\omega}\tau})}$, for which if $\hat{\Phi}=(f_0,g_0,\eta^0)\in C([-\tau,0]; \hat{\He})$ satisfies	
		\begin{multline}\label{visco_smallnessID}
			\|u_1\|^2 + (1-\tilde{k})\|\sqrt{A} u_0\|^2 
			+\int^{0}_{-\tau}
			\Vert \nabla \psi_2(u_0(s))\Vert^2 ds+\\ \int_{0}^{+\infty} k(s)\|\sqrt{A}\eta^0(s)\|^{2} ds+ \int_{-\tau}^0 |\alpha(s+\tau)| \cdot \|B^* g_0(s)\|^2 ds< \rho^2
		\end{multline}
		and
		\begin{equation}\label{visco_additional}
			\max _{s\in [-\tau,0]} \Vert \hat{\Phi}(s)\Vert_{\hat{\mathcal H}}<\rho,
	\end{equation}
	then, problem \eqref{visco_P_abstract} has a unique global solution 
	satisfying 
	$$\Vert U(t)\Vert_{\hat{\mathcal H}}\le C_\rho,\quad \forall \ t\ge 0.$$
\end{theorem}
\begin{proof}
	Let $T$ be a fixed time with $T \geq \tau$ such that,
	\begin{equation*}
			C_T := 
			2\hat{M}^2 e^{2\gamma}\max \left\{\left(1+ e^{\tau} \left[b\sigma+\tau\left(1-\tilde{k}\right)^2\right]\right), e^{\hat{\omega}\tau}\right \}\left(1 + e^{2\hat{\omega} \tau} (\sigma+\tau)^2 \right) e^{-(\hat{\omega} - \omega') T}
			< 1. 
	\end{equation*}
	Also, let us denote
	\begin{equation*}
			C^*_T :=\sup \left\{e^{2T+2b^2\int_{nT}^{(n+1)T} \left(\vert\alpha(s)\vert+\vert\alpha(s+\tau)\vert\right) ds},\; n\in \mathbb{N} \right\}.
	\end{equation*}
	Condition \eqref{cdt_alpha} guarantees that $C_T^*$ is finite. Note that $C_T^* \geq C(T)$, where $C(T)$ is defined in \eqref{constant_t}. 
	Next, let $\rho > 0$ be chosen such that
	\begin{equation}\label{visco_rho_condition}
			\rho \leq \frac{\sqrt{1-\tilde{k}}}{2 \sqrt{2C_T^*}} h^{-1}\left(\frac{1-\tilde{k}}{2}\right) \quad\text{and} \quad L\left( \frac{2 \sqrt{2}}{\sqrt{1-\tilde{k}}}\sqrt{C^*_T} \rho \right) < \frac{1-\tilde{k}}{2}, 
	\end{equation}
	and assume the conditions
	\begin{multline*}
		\|u_1\|^2 + (1-\tilde{k})\|\sqrt{A} u_0\|^2 + \int_{0}^{+\infty} k(s)\|\sqrt{A}\eta^0(s)\|^{2} ds \\
		+\int^{0}_{-\tau}
		\Vert \nabla \psi_2(u_0(s))\Vert^2 ds+\int_{-\tau}^0 |\alpha(s+\tau)| \cdot \|B^* g_0(s)\|^2 ds< \rho^2,
	\end{multline*}
	and
	$$\max _{s\in [-\tau,0]} \Vert \hat{\Phi}(s)\Vert_{\hat{\mathcal H}}<\rho.$$
	This implies, by considering the abstract setting, i.e., the system \eqref{visco_P_abstract},
	\begin{equation}\label{visco_rho_abstract}
		\|U_0\|_{\hat{\He}}^2 +\int_{0}^{\tau}
		\Vert \nabla \psi_2(u_0(s-\tau))\Vert^2 ds + \int_0^\tau |\alpha(s)| \cdot \| B^* g_0(s-\tau)\|^2 \, ds < \rho^2.
	\end{equation}
	
	From Lemma \ref{visco_lem_abst_exis}, we know that there exists a solution \( u \) to problem \eqref{visco_p} on a time interval \( [0,\delta) \). Now, we have
	\begin{equation}\label{visco_cdt_psi_1}
		h_1\left(\left\|\sqrt{A} u_0\right\| \right) \leq h_1\left(\frac{\rho}{\sqrt{1-\tilde{k}}}\right) \leq h_1\left( \frac{1}{2\sqrt{2C^*_T}} h_1^{-1}\left(\frac{1-\tilde{k}}{2}\right) \right) < \frac{1-\tilde{k}}{2},
	\end{equation}
	
	using the fact that \( C^*_T > 1 \). Thus, by Lemma \ref{visco_exis_lemma}, \( \hat{E}(0) > 0 \). Moreover, from \eqref{visco_cdt_psi_1}, we obtain
	\begin{eqnarray*}
		\hat{E}(0) &\leq &\frac{1}{2} \|u_1\|^2 + \frac{3}{4} (1-\tilde{k})\|\sqrt{A} u_0\|^2 + \frac{1}{2} \int_{0}^{+\infty} k(s) \|\sqrt{A} \eta^0(s)\|^2\, ds \\
		& & +\frac{1}{2}\int^{0}_{-\tau}
		\Vert \nabla \psi_2(u_0(s))\Vert^2 ds + \frac{1}{2} \int_{-\tau}^0 |\alpha(s+\tau)| \cdot \|B^* g_0(s)\|^2 \, ds \\
		&<& \rho^2,
	\end{eqnarray*}
	which gives us
	\begin{equation}\label{visco_h_cdt_1_2}
		h_1\left(2\frac{ \sqrt{2C^*_T \hat{E}(0)}}{\sqrt{1-\tilde{k}}} \right) < h_1\left( 2\frac{\sqrt{2C^*_T} }{\sqrt{1-\tilde{k}}}\rho \right) < h_1\left( h_1^{-1}\left( \frac{1-\tilde{k}}{2} \right) \right) = \frac{1-\tilde{k}}{2},
	\end{equation}
	and
\begin{equation}\label{visco_suL}
			L\left(\frac{2 \sqrt{2}}{\sqrt{1-\tilde{k}}}\sqrt{C^*_T \hat{E}(0)}\right) \leq L\left( \frac{2 \sqrt{2}}{\sqrt{1-\tilde{k}}}\sqrt{C^*_T} \rho \right) < \frac{1-\tilde{k}}{2}.
		\end{equation}

	Therefore, by applying Lemma \ref{visco_exis_lemma} once more, we can conclude that \eqref{visco_estimate_psi2} and \eqref{visco_estimate_energy_existance} hold for all \( t \in [0, \delta) \). Using Proposition \ref{visco_pro_1}, we obtain
	\begin{multline}\label{visco_eq_E_bound}
		\qquad 0< \frac{1}{4} \| u_t(t) \|^2 + \frac{1-\tilde{k}}{4} \| \sqrt{A} u(t) \|^2 +\frac{1}{4}\int^{t}_{t-\tau}
		\Vert \nabla \psi_2(u(s))\Vert^2 ds\\ + \frac{1}{4} \int_{t - \tau}^t |\alpha(s+\tau)| \| B^* u_t(s) \|^2 ds + \frac{1}{4} \int_{0}^{+\infty} k(s) \|\sqrt{A} \eta^t(s)\|^2\, ds\leq \hat{E}(t)\leq C^*_T \hat{E}(0), 
	\end{multline}
	for any \( t \in [0, \delta) \). Then, we can extend the solution on the entire interval \( [0, T] \). From \eqref{visco_h_cdt_1_2} and \eqref{visco_eq_E_bound}, for \( t = T \), we have
	\begin{equation}\label{visco_sqrt_A_u_t}
		h_1\left( \| \sqrt{A} u(T)\| \right) \leq h_1(2\frac{\sqrt{\hat{E}(T)}}{\sqrt{1-\tilde{k}}}) \leq h_1\left( 2\frac{\sqrt{2C^*_T \hat{E}(0)}}{\sqrt{1-\tilde{k}}} \right) \leq h_1\left(2 \frac{\sqrt{2C^*_T}}{\sqrt{1-\tilde{k}}} \rho \right) < \frac{1-\tilde{k}}{2}.
	\end{equation}
	
	Moreover, from the assumption \eqref{visco_rho_abstract} on the initial data $(U_0, \eta^0)$ and
	from \eqref{visco_eq_E_bound}, we arrive at
	\[
	\frac{1}{4} \|U(t)\|_{\hat{\He}}^2 \leq \hat{E}(t) \leq C^*_T \hat{E}(0) < C^*_T \rho^2, \quad \forall t \in [0, T],
	\]
	and so
	\begin{equation}\label{visco_U_rdo_leq}
		\|U(t)\|_{\hat{\He}} \leq C_{\rho} := 2\sqrt{C^*_T} \rho,
	\end{equation}
	for any \( t \in [0, T] \). By choosing eventually smaller values of \( \rho \), we assume that \( \rho \) is such that \( L(C_\rho) < \frac{\hat{\omega} - \omega'}{2\hat{M}(1+e^{\hat{\omega}\tau})} \). Thus, the well-posedness assumption is satisfied on the interval \( [0, T] \). Applying Theorem \ref{visco_th_stability}, we then obtain the following estimate:
	\begin{equation}\label{visco_eq_23}
		\|U(t)\|_{\hat{\He}} \leq \hat{M} e^{\gamma} \left( \|U_0\|_{\hat{\He}} +\hat{\beta} + \int_0^\tau e^{\hat{\omega} s} |\alpha(s)| \cdot \| \hat{\Phi}(s-\tau)\|_{\hat{\He}} ds \right) e^{-\frac{\hat{\omega} - \omega'}{2} t}, 
	\end{equation}
	for any $t \in [0, T]$. By using assumption \eqref{cdt_alpha}, \eqref{visco_additional} and H\"{o}lder inequality, we get
	\[
	\int_0^{\tau} |\alpha(s)| e^{\hat{\omega} s} \| \hat{\Phi}(s-\tau)\|_{\hat{\He}} ds \leq e^{\hat{\omega} \tau} \sqrt{\int_0^{\tau} |\alpha(s)| ds }\sqrt{\int_0^{\tau} |\alpha(s)| \cdot \| \hat{\Phi}(s-\tau)\|_{\hat{\He}}^2 ds }\leq e^{\hat{\omega} \tau}\sigma \rho.
	\]
	According to \eqref{visco_beta_eq}, \eqref{visco_additional} and \eqref{visco_suL}, we get
	$$\hat{\beta}\le \frac 12 \int_0^\tau e^{\hat{\omega} s} \Vert \hat{\Phi} (s-\tau)\Vert_{\hat{\mathcal{H}}} ds \le \frac 12 e^{\hat{\omega}\tau}\rho \tau.
	$$
	From \eqref{visco_eq_23}, we arrive at
	\begin{eqnarray}\label{visco_eq_24}\notag
		\|U(t)\|^2_{\hat{\He}} &\leq& \hat{M}^2 e^{2\gamma} \left( \|U_0\|_{\hat{\He}} + e^{\hat{\omega} \tau} \rho (\sigma+\tau) \right)^2 e^{-(\hat{\omega} - \omega')t}, \\
		&\leq&2 \hat{M}^2 e^{2\gamma} \left( \|U_0\|_{\hat{\He}}^2 + e^{2\hat{\omega} \tau} (\sigma+\tau)^2 \rho^2 \right) e^{-(\hat{\omega} - \omega')t}, 
	\end{eqnarray}
	which, by using \eqref{visco_rho_abstract}, leads to
	\begin{equation}\label{visco_U_estimate_bound_above}
		\|U(t)\|^2_{\hat{\He}} \leq 2\hat{M}^2 \rho^2 e^{2\gamma} \left(1 + e^{2\hat{\omega} \tau} (\sigma+\tau)^2 \right) e^{-(\hat{\omega} - \omega') t}, 
	\end{equation}
	for all \( t \in [0, T] \). Also, for all \( s \in [T - \tau, T] \subseteq [0, T] \), \eqref{visco_U_estimate_bound_above} yields
		\begin{eqnarray*}
			\Vert u_t(s)\Vert^2 &\leq & \|U(s)\|_{\hat{\mathcal{H}}}^2 \leq 2\hat{M}^2 \rho^2 e^{2\gamma} \left(1 + e^{2\hat{\omega} \tau} (\sigma+\tau )^2 \right) e^{-(\hat{\omega} - \omega') s} \\
			&\leq & 2\hat{M}^2 \rho^2 e^{2\gamma+\hat{\omega}\tau} \left(1 + e^{2\hat{\omega} \tau} (\sigma+\tau)^2 \right) e^{-(\hat{\omega} - \omega') T}, 
		\end{eqnarray*}
	Consequently, by \eqref{cdt_alpha},
		\begin{equation}\label{visco_nuova1}
			\begin{array}{l}
				\displaystyle{ \int_{T-\tau}^{T} |\alpha(s)| \cdot \|B^* u_t(s)\|^2 ds \leq 2\hat{M}^2 \rho^2 e^{2\gamma+\hat{\omega}\tau} b\left(1 + e^{2\hat{\omega} \tau} (\sigma+\tau)^2 \right) e^{-(\hat{\omega} - \omega') T} \int_{T-\tau}^{T} |\alpha(s)| ds} \\
				\displaystyle{ \hspace{4.3 cm}\leq 2\hat{M}^2 \rho^2 e^{2\gamma+\hat{\omega}\tau} b\sigma\left(1 + e^{2\hat{\omega} \tau} (\sigma+\tau)^2 \right) e^{-(\hat{\omega} - \omega') T}. }
			\end{array}
		\end{equation}
	Moreover, from
		\begin{equation*}
			\int^{T}_{T-\tau} \Vert \nabla \psi_2(u(s))\Vert^2 ds \leq \int^{T}_{T-\tau} L^2\left(\left\|\sqrt{A}u(s)\right\|\right) \left\|\sqrt{A}u(s)\right\|^2ds,
		\end{equation*}
		since we have
		$$ L\left(\Vert \Ar u(s)\Vert\right) \leq L\left(\frac{2}{\sqrt{1-\tilde{k}}}{\sqrt{\hat{E}(s)}} \right)\leq L\left(\frac{2\sqrt{2}}{\sqrt{1-\tilde{k}}}\sqrt{C^*_T\hat{E}(0)}\right) <\frac{1-\tilde{k} }{2} $$
		and
		\begin{eqnarray*}
			(1-\tilde{k})\left\|\sqrt{A}u(s)\right\|^2 &\leq & \|U(s)\|_{\hat{\mathcal{H}}}^2 \leq 2\hat{M}^2 \rho^2 e^{2\gamma} \left(1 + e^{2\hat{\omega} \tau} (\sigma+\tau)^2 \right) e^{-(\hat{\omega} - \omega') s} \\
			&\leq & 2\hat{M}^2 \rho^2 e^{2\gamma+\hat{\omega}\tau} \left(1 + e^{2\hat{\omega} \tau} (\sigma+\tau)^2 \right) e^{-(\hat{\omega} - \omega') T}, 
		\end{eqnarray*}
		it follows
		\begin{equation}\label{visco_cdt_bound_psi}
			\int^{T}_{T-\tau} \Vert \nabla \psi_2(u(s))\Vert^2 ds \leq 2\hat{M}^2 \rho^2 e^{2\gamma+\hat{\omega}\tau} \tau\left(1-\tilde{k}\right)\left(1 + e^{2\hat{\omega} \tau} (\sigma+\tau)^2 \right) e^{-(\hat{\omega} - \omega') T} .
		\end{equation}

	This last fact together with \eqref{visco_U_estimate_bound_above} implies that
	\begin{multline}\label{visco_U_T_rho}
		\|U(T)\|_{\hat{\mathcal{H}}}^2 + \int^{T}_{T-\tau} \Vert \nabla \psi_2(u(s))\Vert^2 ds + \int_{T-\tau}^{T} |\alpha(s)| \cdot \|B^* u_t(s)\|^2 \, ds \\
		\leq 2\hat{M}^2 \rho^2 e^{2\gamma}\left(1+ e^{\hat{\omega}\tau} \left[b\sigma+\tau \left(1-\tilde{k}\right)\right]\right)\left(1 + e^{2\hat{\omega} \tau} \left(\sigma+\tau\right)^2 \right) e^{-(\hat{\omega} - \omega') T} \leq C_T \rho^2 < \rho^2.
	\end{multline}
		Moreover, from \eqref{visco_eq_24}, we have
		\begin{equation}\label{visco_U_rho}
			\max_{s\in [T-\tau, T]} \|U(s)\|_{\hat{\mathcal{H}}}^2< 
			2\hat{M}^2 \rho^2 e^{2\gamma+\hat{\omega}\tau} \left(1 + e^{2\hat{\omega} \tau }(\sigma+\tau)^2 \right) e^{-(\hat{\omega} - \omega') T}\le C_T\rho^2<\rho^2.
	\end{equation}
Conditions \eqref{visco_U_T_rho} and \eqref{visco_U_rho} allow us to replicate the previous arguments on the interval \([T, 2T]\).

 Specifically, we consider the initial value problem
	\begin{equation}\label{visco_P_abstract_222}
		\left\{
		\begin{array}{ll}
			W_t(t) = \hat{\A} W(t) +\alpha(t)\hat{\B} W(t - \tau)+ \hat{F}(W(t), W(t - \tau)), &\quad t \in [T, 2T], \\
			W(s) = U(s), \quad s \in [T - \tau, T],
		\end{array}
		\right.
	\end{equation}
	where \(U(\cdot)\) denotes the solution to \eqref{visco_P_abstract} on \([0, T]\).
	Now, let introduce the energy functional for the solution:
	\begin{eqnarray}\label{visco_function_E_222} \notag \hat{\mathcal{E}}(t)&=&\frac{1}{2}\left\|w_{t}(t)\right\|^{2}+\frac{ 1 - \tilde{k}}{2}\left\|\sqrt{A}w(t)\right\|^{2}-\psi_1(w(t))+\frac{1}{2}\int^{t}_{t-\tau}
		\Vert \nabla \psi_2(w(s))\Vert^2 ds \\
		& &+ \frac{1}{2} \int_{0}^{+\infty} k(s) \|\sqrt{A} \eta^t(s)\|^2\, ds+ \frac{1}{2}\int_{t-\tau}^{t}\vert \alpha(s+\tau)\vert \Vert B^*w_t(s) \Vert^2ds , 
	\end{eqnarray}
	where $\eta^t(s) := w(t) - w(t-s)$. Observe that $\hat{\mathcal{E}}(T) = \hat{E}(T)$. By Lemma \ref{visco_lem_abst_exis}, problem \eqref{visco_P_abstract_222} with initial data \(U(s)\), \(s \in [T - \tau, T]\), admits a unique local solution \(W(\cdot)\) on \([T, T + \delta)\) given by Duhamel's formula:
	\begin{equation}\label{visco_Formule_W_2T}
		W(t) = \hat{S}(t - T) U(T) + \int_T^t \hat{S}(t - T - s) \bigl[\alpha(s)\hat{\B} W(s - \tau)+ \hat{F}(W(s), W(s - \tau))\bigr] \, ds,
	\end{equation}
	for all \(t \in [T, T + \delta)\). We may assume \(\delta \leq \tau\) and that \(W\) is nontrivial; otherwise, \eqref{visco_U_rho} holds trivially for all \(t \geq T\).
	From \eqref{visco_sqrt_A_u_t}, we have \(\hat{\mathcal{E}}(T) > 0\). If
	$$ \hat{\mathcal{E}}(t) \geq \frac{1}{4} \|w_t(t)\|^2+\frac{1-\tilde{k}}{4} \|\sqrt{A} w(t)\|^2,$$
	and 
		$\Vert \nabla\psi_2(w(t))\Vert\le(1-\tilde{k}) \Vert \sqrt{A}w(t) \Vert$, for all \(t \in [T, T + \delta)\), then, following Proposition \ref{visco_pro_1}, we obtain the growth estimate
	\begin{equation}\label{visco_estiamte_energy_22}
	\hat{\mathcal {E}}(t)\leq e^{2\int_{T}^{t}[ 1+ \left(\vert\alpha(s)\vert+\vert\alpha(s+\tau)\vert\right)b] ds}\hat{\mathcal{E}}(T), \quad \forall t \in [T, T + \delta) ,
	\end{equation}
	Additionally, by applying the same reasoning as in Lemma \ref{visco_exis_lemma} and noting that $$C^*_T \geq e^{2\int_{T}^{2T}[ 1+ \left(\vert\alpha(s)\vert+\vert\alpha(s+\tau)\vert\right)b] ds},$$ if
\begin{equation}\label{visco_cdt_h_T_2T}
		h_1\left(\frac{2\sqrt{2} }{\sqrt{1-\tilde{k}}} (C_T^*)^{1/2}{\hat{\mathcal E}}^{1/2}(T)\right) <\frac{1-\tilde{k}}{2} \quad \text{and} \quad L\left(\frac{2\sqrt{2} }{\sqrt{1-\tilde{k}}} (C_T^*)^{1/2}{\hat{\mathcal E}}^{1/2}(T)\right) <\frac{1-\tilde{k}}{2},
	\end{equation}
	thus, for every $t \in [T, T + \delta)$,
	\begin{eqnarray}\label{visco_function_E_333} \notag
		\hat{\mathcal{E}}(t)&>& \frac{1}{4}\left\|w_{t}(t)\right\|^{2}+\frac{1-\tilde{k}}{4}\left\|\sqrt{A}w(t)\right\|^{2}+\frac{1}{4}\int^{t}_{t-\tau}
		\Vert \nabla \psi_2(w(s))\Vert^2 ds \\
		& &+ \frac{1}{4} \int_{0}^{+\infty} k(s) \|\sqrt{A} \eta^t(s)\|^2\, ds + \frac{1}{4}\int_{t-\tau}^{t}\vert \alpha(s+\tau)\vert \Vert B^*w_t(s) \Vert^2ds ,
	\end{eqnarray}
	and 
		\begin{equation}\label{visco_quasi}
			\Vert \nabla\psi_2(w(t))\Vert\le(1-\tilde{k}) \Vert \sqrt{A}w(t) \Vert.
		\end{equation}
	which leads to 
	\begin{equation}\label{visco_E_W_cdt}	\hat{\mathcal{E}}(t)>\frac{1}{4}\left\|W(t)\right\|_{\hat{\He}}^{2}, \quad \forall t \in [T, T + \delta).
	\end{equation} 
	Note that, by \eqref{visco_U_T_rho}, we have $\hat{\mathcal{E}}(T)<\rho$, thus, the condition \eqref{visco_cdt_h_T_2T} is satisfied. Consequently, \eqref{visco_function_E_333}, \eqref{visco_E_W_cdt} and \eqref{visco_quasi} hold for all $t \in [T, T + \delta)$.
	As a result, the inequality \eqref{visco_estiamte_energy_22} is fulfilled. Combining \eqref{visco_estiamte_energy_22} and \eqref{visco_function_E_333}, we finally get
	\begin{multline}\label{visco_eq_E_bound_T}
		\frac{1}{4} \|w_t(t)\|^2 + \frac{1}{4} \|\sqrt{A} w(t)\|^2 +\frac{1}{4}\int^{t}_{t-\tau}
		\Vert \nabla \psi_2(w(s))\Vert^2 ds 
		+ \frac{1}{4} \int_{t - \tau}^t |\alpha(s+\tau)| \cdot \|B^* w_t(s)\|^2 ds\\+ \frac{1}{4} \int_{0}^{+\infty} k(s) \|\sqrt{A} \eta^t(s)\|^2\, ds<\hat{\mathcal{E}}(t) \leq e^{2\int_{T}^{2T}[ 1+ \left(\vert\alpha(s)\vert+\vert\alpha(s+\tau)\vert\right)b] ds} \hat{\mathcal{E}}(T) \leq C_T^* \hat{\mathcal{E}}(T).
	\end{multline}
	Therefore, the solution $W$ remains bounded, allowing us to extend it to $t = T + \delta$ and across the entire interval $[T, 2T]$. Furthermore, the estimates \eqref{visco_eq_E_bound_T} and \eqref{visco_E_W_cdt} remain valid for every $t$ in $[T, 2T]$. Consequently, we obtain
	\begin{equation}\label{visco_W_rho_leq}
		\|W(t)\|_{\hat{\He}} \leq 2\sqrt{C_T^* \hat{\mathcal{E}}(T)} \leq 2\sqrt{C_T^*} \rho = C_\rho.
	\end{equation}
	By combining the two partial solutions \eqref{visco_formule_U_rho} and \eqref{visco_Formule_W_2T} obtained on the time intervals $[0, T]$ and $[T, 2T]$, respectively, we establish the existence of a unique solution $U \in C([0, 2T]; \hat{\He})$ to \eqref{visco_P_abstract} defined on $[0, 2T]$ and satisfying Duhamel's formula \eqref{visco_formule_U_rho} for all $t \in [0, 2T]$.

Additionally, from \eqref{visco_U_rdo_leq} and \eqref{visco_W_rho_leq}, the solution $U$ satisfies \eqref{visco_U_rho_th} on the interval $[0, 2T]$. Hence, since \( L(C_\rho) < \frac{\hat{\omega} - \omega'}{2\hat{M}(1+e^{\hat{\omega}\tau})} \), the exponential decay estimate \eqref{visco_stab_estimate_2} holds for $U$; namely, 
	\begin{equation}\label{visco_stab_estimate_2_needed}
		\left\Vert U(t)\right\Vert_{\hat{\He}} \leq \hat{M}e^\gamma\left(\left\Vert U_0\right\Vert_{\hat{\He}}+\hat{\beta}+\int_{0}^{\tau} e^{\hat{\omega} s}\vert \alpha(s)\vert\cdot\Vert \hat{\Phi}(s-\tau)\Vert_{\hat{\He}} ds \right) e^{-\frac{\hat{\omega} - \omega'}{2} t}.
	\end{equation}
	for all $t \in [0, 2T]$. Again, we deduce that \eqref{visco_U_estimate_bound_above} holds for all $t \in [0, 2T]$. Furthermore, as before, we can obtain \eqref{visco_nuova1} and \eqref{visco_cdt_bound_psi} on $[2T - \tau, 2T].$

	Then, similar to \eqref{visco_U_T_rho}, we get
	\begin{multline}\label{visco_U_2T_rho}
		\|U(2T)\|_{\hat{\mathcal{H}}}^2 + \int^{2T}_{2T-\tau} \Vert \nabla \psi_2(u(s))\Vert^2 ds + \int_{2T-\tau}^{2T} |\alpha(s)| \cdot \|B^* u_t(s)\|^2 \, ds \\
		\leq 2\hat{M}^2 \rho^2 e^{2\gamma}\left(1+ e^{\hat{\omega}\tau} \left[b\sigma+\tau \left(1-\tilde{k}\right)\right]\right)\left(1 + e^{2\hat{\omega} \tau} \left(\sigma+\tau\right)^2 \right) e^{-(\hat{\omega} - \omega') T} \leq C_T \rho^2 < \rho^2.
	\end{multline}
	Moreover,
		\begin{equation}\label{visco_U_rho_2}
		\max_{s\in [2T-\tau, 2T]} \|U(s)\|_{\hat{\mathcal{H}}}^2< 2\hat{M}^2 \rho^2 e^{2\gamma+\hat{\omega}\tau} \left(1 + e^{2\hat{\omega} \tau }(\sigma+\tau)^2 \right) e^{-(\hat{\omega} - \omega') T}\le C_T\rho^2<\rho^2.
		\end{equation}
	The inequalities \eqref{visco_U_2T_rho} and \eqref{visco_U_rho_2} enable us to repeat the same reasoning used earlier on the interval $[2T, 3T]$. This yields the existence of a unique solution to \eqref{visco_P_abstract} defined on $[0, 3T]$ that satisfies \eqref{visco_formule_U_rho} for every $t \in [0, 3T]$. By iterating this process, we ultimately obtain a unique global solution $U \in C([0, +\infty); \hat{\He})$ to \eqref{visco_P_abstract}, which fulfills \eqref{visco_formule_U_rho}.

	Consequently, we have shown that for sufficiently small initial data, solutions to problem \eqref{visco_P_abstract} exist globally in time and are bounded by a positive constant $C_\rho$ satisfying \( L(C_\rho) < \frac{\hat{\omega} - \omega'}{2\hat{M}(1+e^{\hat{\omega}\tau})} \).
\end{proof}

As a direct consequence of Theorems \ref{visco_Stability} and \ref{visco_th_stability}, we obtain the following energy decay result for the viscoelastic model \eqref{visco_p} under small initial data.	
		\begin{theorem}\label{last_visco}
			Assume \eqref{visco_cdt_alpha_2} and \eqref{cdt_alpha}. There exists $\rho>0$ 
			such that, if the initial datum $\hat{\Phi}=(f_0,g_0,\eta^0)\in C([-\tau,0]; \hat{\He})$ satisfies
			\begin{multline*}
				\|u_1\|^2 + (1-\tilde{k})\|\sqrt{A} u_0\|^2 
				+\int^{0}_{-\tau}
				\Vert \nabla \psi_2(u_0(s))\Vert^2 ds+\\ \int_{0}^{+\infty} k(s)\|\sqrt{A}\eta^0(s)\|^{2} ds+ \int_{-\tau}^0 |\alpha(s+\tau)| \cdot \|B^* g_0(s)\|^2 ds< \rho^2.
			\end{multline*}
			and
			$$
			\max _{s\in [-\tau,0]} \Vert \hat{\Phi}(s)\Vert_{\hat{\mathcal H}}<\rho,
			$$
			then
			problem \eqref{visco_p} has a unique global solution 
			satisfying 
			$$\hat{E}(t)\le \tilde C e^{-(\hat{\omega}-\omega') t},\quad \forall \ t\ge 0,$$
			where $\tilde C$ is a suitable constant depending on the initial data.
		\end{theorem}
\section{Examples}\label{expls_sec}
We now present some concrete examples that fall within the abstract frameworks developed in Sections \ref{linear_dam_sec} and \ref{visco_dam_sec}. In each case, we show how the abstract operators and assumptions are realized, and we discuss the physical interpretation of the models.
\subsection{Euler--Bernoulli beam with Duffing nonlinearities
and power-type delayed force}\label{ex:beam-duffing}

Let \(\Omega = (0,L)\). Consider the following beam equation
\begin{equation} \label{beam}
\left\{
\begin{array}{ll}
u_{tt} + u_{xxxx} + c u_t + \alpha(t) b(x) u_t(x,t-\tau) = \mu |u|^{p-2}u + \nu |u(t-\tau)|^{q-2}u(t-\tau), \\
(x,t) \in (0,L) \times (0,+\infty), \\
u(x,t-\tau) = f_0(x,t-\tau), \quad u_t(x,t-\tau) = g_0(x,t-\tau), \quad t \in [0,\tau],\\
u(0,t) = u_x(0,t) = u(L,t) = u_x(L,t) = 0, \quad t > 0,
\end{array}
\right.
\end{equation}
with \(u_0 := f_0(0)\), \(u_1 := g_0(0)\), \(p,q > 2\), \(\mu,\nu > 0\), \(c > 0\), and \(b \in L^\infty(0,L)\)  {satisfies $b(x)\ge 0$ a.e. $x\in\Omega.$}

The above system falls into the form \eqref{p} with \(H = L^2(0,L)\), \(A = \partial_x^4\) and \(D(A) = H^4(0,L) \cap H_0^2(0,L)\). Moreover, \(D(A^{\frac{1}{2}}) = H_0^2(0,L)\).

Here \(C = c^{1/2}I\), \(B\) is the multiplication operator by {\(b(x)^{1/2}\),} and the nonlinearities are given by
\[
\psi_1(u) = \frac{\mu}{p}\int_0^L |u|^p\,dx, \qquad
\psi_2(u) = \frac{\nu}{q}\int_0^L |u|^q\,dx.
\]

The Fr\'{e}chet derivatives are
\begin{equation}
\nabla \psi_1(u) = \mu |u|^{p-2}u, \qquad \nabla \psi_2(u) = \nu |u|^{q-2}u.
\end{equation}

Let us consider the functionals
\begin{equation}
\psi_i(u) := \frac{\mu_i}{p_i} \int_0^L |u|^{p_i} dx, \quad \forall u \in H_0^2(0,L), \quad i = 1,2,
\end{equation}

with \(\mu_1 = \mu, \mu_2 = \nu\), \(p_1 = p, p_2 = q\). By Sobolev's embedding theorem \(H_0^2(0,L) \hookrightarrow L^\infty(0,L)\), the functionals \(\psi_i\) are well-defined for all \(p_i > 2\). Also, the Gateaux derivative of \(\psi_i\) at any point \(u \in H_0^2(0,L)\) is

\begin{equation}
D\psi_i(u)(v) = \mu_i \int_0^L |u|^{p_i-2} u v dx,
\end{equation}

for all \(v \in H_0^2(0,L)\). Moreover, as in the general setting, if \(p_i > 2\), then \(\psi_i\) satisfies the assumptions (A1), (A2), (A3).
In order to reformulate \eqref{beam} as an abstract first-order equation, we introduce the Hilbert space \(\mathcal{H} = H_0^2(0,L) \times L^2(0,L)\) endowed with the inner product

\begin{equation}
\left\langle \begin{pmatrix} u_1 \\ v_1 \end{pmatrix}, \begin{pmatrix} u_2\\ v_2 \end{pmatrix} \right\rangle_{\mathcal{H}} := \int_0^L \nabla^2 u_1 \nabla^2 u_2 \, dx + \int_0^L  v_1  v_2 \, dx.
\end{equation}

We set \(U = (u, u_t)^T\). Then, \eqref{beam} can be rewritten in the form
\begin{equation}
U'(t) = \A U(t) + \alpha(t)\B U(t-\tau) + F(U(t), U(t - \tau)),
\end{equation}
where
\begin{equation}
\A \begin{pmatrix} u \\ u_t \end{pmatrix}  = \begin{pmatrix} u_t\\ -u_{xxxx} -cu_t \end{pmatrix}, \qquad \B  \begin{pmatrix} u \\ u_t \end{pmatrix} = \begin{pmatrix}  0 \\- b(x)u_t \end{pmatrix},
\end{equation}
with domain
\begin{equation}
D(\A) = \{(u,v) \in H_0^2(0,L) \times H_0^2(0,L) : u \in H^4(0,L) \cap H_0^2(0,L)\},
\end{equation}
and
\begin{equation}
 F(U(t), U(t - \tau)) = \begin{pmatrix}
0 \\
\mu |u|^{p-2}u + \nu |u(t - \tau)|^{q-2}u(t - \tau)
\end{pmatrix}.
\end{equation}
Let define the energy functional as follows 
\begin{eqnarray*} 
E(t) &= &\frac{1}{2} \int_0^L |u_t(x,t)|^2 dx + \frac{1}{2} \int_0^L |u_{xx}(x,t)|^2 dx - \frac{\mu}{p} \int_0^L |u(x,t)|^p dx \\
& & + \frac{\nu^2}{2} \int_{t-\tau}^t \int_0^L |u(x,s)|^{2q-2} dx \, ds + \frac{1}{2} \int_{t-\tau}^t |\alpha(s+\tau)| \int_0^L b(x) |u_t(x,s)|^2 dx \, ds.
\end{eqnarray*}
This models the active vibration control of a clamped flexible beam (robot arm, aeroelastic panel) made of a nonlinearly elastic material (Duffing restoring force \(\mu |u|^{p-2}u\)), equipped with a delayed distributed actuator with gain \(b(x)\) and latency \(\tau\). The term \(\nu |u(t-\tau)|^{q-2}u(t-\tau)\) represents a nonlinear delayed restoring force (e.g. reaction of a nonlinear spring measured with delay).

{
	Thus, Theorem \ref{last} ensures
	that for sufficiently small initial data
	solutions to \eqref{beam} are globally defined and satisfy an exponential decay estimate provided that \eqref{cdt_alpha_2} is satisfied. }
\subsection{Viscoelastic wave equation with seismic isolation,
sine source and van der Pol delayed force}

Let \(\Omega \subset \mathbb{R}^n\) be a bounded domain. Consider
\begin{equation} \label{wave_visco}
\left\{
\begin{array}{ll}
u_{tt} - \Delta u + \int_{0}^{+\infty} k(s) \Delta u(t-s) ds + \alpha(t) \mathbf{1}_{\chi} u_t(t-\tau) = \sin u + \mu u^3(t-\tau) - \lambda u(t-\tau), \\
(x,t) \in \Omega \times (0,+\infty), \\
u = 0 \quad \text{on } \partial\Omega \times (0,+\infty), \\
u(t-\tau) = f_0(t-\tau), \quad t \in (-\infty, \tau], \\
u_t(t-\tau) = g_0(t-\tau), \quad t \in [0, \tau],
\end{array}
\right.
\end{equation}
with \(u_0 := f_0(0)\) and \(u_1 := g_0(0)\), where: \(\chi \subset \Omega\) is the control region, \(k(t) = k_0 e^{-\zeta t}\) is the memory kernel with \(\tilde{k} = \frac{k_0}{\zeta} < 1\), \(\mu, \lambda > 0\), \(\mathbf{1}_{\chi}\) is the characteristic function of \(\chi\) and \(B = \mathbf{1}_{\chi}\) (multiplication operator).

The system \eqref{wave_visco} can be recast in the abstract framework \eqref{visco_p} with the choice of \(H = L^2(\Omega)\), \(A = -\Delta\) and \(D(A) = H^2(\Omega) \cap H_0^1(\Omega)\). Moreover, \(D(A^{\frac{1}{2}}) = H_0^1(\Omega)\).

Here \(B = \mathbf{1}_{\chi}\) (the multiplication operator by the characteristic function of \(\chi\)), and the nonlinearities are given by
\begin{equation}
\psi_1(u) = \int_{\Omega} (1 - \cos u) dx, \qquad \psi_2(u) = \frac{\mu}{4} \int_{\Omega} u^4 dx - \frac{\lambda}{2} \int_{\Omega} u^2 dx.
\end{equation}
The Fr\'{e}chet derivatives are
\begin{equation}
\nabla \psi_1(u) = \sin u, \qquad \nabla \psi_2(u) = \mu u^3 - \lambda u.
\end{equation}
Let us consider the functionals
\begin{equation}
\psi_1(u) := \int_{\Omega} (1 - \cos u) dx, \quad \forall u \in H_0^1(\Omega),
\end{equation}
and
\begin{equation}
\psi_2(u) := \frac{\mu}{4} \int_{\Omega} u^4 dx - \frac{\lambda}{2} \int_{\Omega} u^2 dx, \quad \forall u \in H_0^1(\Omega).
\end{equation}

By Sobolev's embedding theorem \(H_0^1(\Omega) \hookrightarrow L^6(\Omega)\) for \(n \leq 3\) (physical dimension), the functionals are well-defined. Also, the Gateaux derivatives are

\begin{equation}
D\psi_1(u)(v) = \int_{\Omega} \sin u \cdot v \, dx, \quad D\psi_2(u)(v) = \int_{\Omega} (\mu u^3 - \lambda u) v \, dx,
\end{equation}

for all \(v \in H_0^1(\Omega)\). Moreover, as in the general setting, \(\psi_1\) satisfies (A1)-(A3), while \(\psi_2\) satisfies (A1)-(A3) under a smallness condition on \(\lambda\).

Let us define \(\eta_s^t\) as in \eqref{def_eta}. Then, system \eqref{wave_visco} can be rewritten as follows
\begin{multline} \label{wave_visco_rewritten}
\qquad u_{tt}(x,t) - (1 - \tilde{k})\Delta u(x,t) - \displaystyle\int_0^{+\infty} k(s)\Delta \eta^t(x,s) ds + \alpha(t) \mathbf{1}_{\chi} u_t(x,t-\tau) \\
  = \sin u(x,t) + \mu u^3(x,t-\tau) - \lambda u(x,t-\tau), \qquad
\end{multline}
for all $(x,t) \in \Omega \times (0,+\infty)$, with
\begin{equation}
\left\{
\begin{array}{ll}
\eta_t^t(x,s) = -\eta_s^t(x,s) + u_t(x,t), & \quad (x,s,t) \in \Omega \times (0,+\infty) \times (0,+\infty), \\
u(x,t) = 0, & \quad (x,t) \in \partial\Omega \times (0,+\infty), \;  \\
 \eta^t(x,s) = 0, & \quad  (x,s) \in \partial\Omega \times (0,+\infty), \; t \geq 0, \\
u(x,0) = u_0(x) := f_0(x,0), \quad u_t(x,0) = u_1(x) := g_0(x,0), & \quad  x \in \Omega, \\
\eta^0(x,s) = \eta_0(x,s) := f_0(x,0) - f_0(x,-s), & \quad (x,s) \in \Omega \times (0,+\infty), \\
u_t(x,t) = g_0(x,t), & \quad (x,t) \in \Omega \times [-\tau, 0].
\end{array}
\right.
\end{equation}

In order to reformulate \eqref{wave_visco} as an abstract first-order equation, we introduce the Hilbert space \(L_k^2((0,+\infty); H_0^1(\Omega))\) endowed with the inner product
\begin{equation}
\langle \phi, \varphi \rangle_{L_k^2((0,+\infty); H_0^1(\Omega))} := \int_{\Omega} \left( \int_0^{+\infty} k(s) \nabla \phi(x,s) \nabla \varphi(x,s) ds \right) dx,
\end{equation}
and the Hilbert space
\begin{equation}
\hat{\He} = H_0^1(\Omega) \times L^2(\Omega) \times L_k^2((0,+\infty); H_0^1(\Omega)).
\end{equation}
We set \(U = (u, u_t, \eta^t)\). Then, \eqref{wave_visco_rewritten} can be rewritten in the abstract form \eqref{visco_P_abstract}, where
\begin{equation}
\hat{\A}
\begin{pmatrix} u \\ v \\ w \end{pmatrix}
=
\begin{pmatrix}
v \\
(1 - \tilde{k}) \Delta u + \displaystyle\int_0^{+\infty} k(s) \Delta w(s) \, ds \\
- w_s + v
\end{pmatrix},
\end{equation}
with domain
\begin{equation}
D(\hat{\A}) = \bigg\{ (u,v,w) \in H_0^1(\Omega) \times H_0^1(\Omega) \times L_k^2((0,+\infty); H_0^1(\Omega)) : \bigg.
\end{equation}
\begin{equation}
\bigg. (1 - \tilde{k})u + \int_0^{+\infty} k(s)w(s)ds \in H^2(\Omega) \cap H_0^1(\Omega), \; w_s \in L_k^2((0,+\infty); H_0^1(\Omega)) \bigg\},
\end{equation}
and
{\begin{equation}
\hat{\B}
\begin{pmatrix} u \\ v \\ w \end{pmatrix}
=
\begin{pmatrix}
0 \\
- \mathbf{1}_{\chi} v \\
0
\end{pmatrix}, \qquad  \hat{\Phi}(t-\tau) = (f_0(t-\tau), g_0(t-\tau), \eta^0)^T \text{ in } [0, \tau],
\end{equation}
and nonnlinear term }
\begin{equation}
\hat{F}(U(t), U(t-\tau))
=
\begin{pmatrix}
0 \\
\sin u(t) + \mu u^3(t-\tau) - \lambda u(t-\tau) \\
0
\end{pmatrix}.
\end{equation}

The energy functional is given by

\begin{eqnarray*} \label{energy_visco_example_5} 
\hat{E}(t) &:= & \frac{1}{2} \int_{\Omega} |u_t(x,t)|^2 dx + \frac{1 - \tilde{k}}{2} \int_{\Omega} |\nabla u(x,t)|^2 dx - \int_{\Omega} (1 - \cos u(x,t)) dx \\
& &  + \frac{1}{2} \int_0^{+\infty} k(s) \int_{\Omega} |\nabla \eta^t(x,s)|^2 dx \, ds  + \frac{1}{2} \int_{t-\tau}^t \int_{\Omega} |\mu u(x,s)^3 - \lambda u(x,s)|^2 dx \, ds \\
& &  + \frac{1}{2} \int_{t-\tau}^t |\alpha(s+\tau)| \int_{\chi} |u_t(x,s)|^2 dx \, ds.
\end{eqnarray*}
{Thus, Theorem \ref{last_visco} ensures
that for sufficiently small initial data
 solutions to \eqref{wave_visco} are globally defined and satisfy an exponential decay estimate provided that \eqref{visco_cdt_alpha_2} is satisfied. }
 
\subsection{Viscoelastic plate with exponential source
and power delayed force}\label{ex:visco-plate}
Let \(\Omega \subset \mathbb{R}^2\) be a bounded domain, \(H = L^2(\Omega)\), \(A = \Delta^2\), \(D(A) = H^4(\Omega) \cap H_0^2(\Omega)\), so that \(D(\sqrt{A}) = H_0^2(\Omega)\). Consider
\begin{equation} \label{plate}
u_{tt} + \Delta^2 u - \displaystyle\int_{0}^{+\infty} k(s) \Delta^2 u(t-s) ds + \alpha(t) \mathbf{1}_{\chi} u_t(t-\tau) 
= (e^u - 1) + \nu |u(t-\tau)|^{q-2}u(t-\tau), 
\end{equation}
for all $(x,t) \in \Omega \times (0,+\infty)$, with
\begin{equation}
\left\{
\begin{array}{ll}
u=\Delta u = 0, & \quad  (x,t) \in \partial\Omega \times (0,+\infty), \\
u(t-\tau) = f_0(t-\tau), &\quad t \in (-\infty, \tau], \\
u_t(t-\tau) = g_0(t-\tau), & \quad t \in [0, \tau],
\end{array}
\right.
\end{equation}
with \(u_0 := f_0(0)\) and \(u_1 := g_0(0)\), where: \(\chi \subset \Omega\) is the control region, \(k(t) = k_0 e^{-\zeta t}\) is the memory kernel with \(\tilde{k} = \frac{k_0}{\zeta} < 1\), \(q > 2\), \(\nu > 0\), \(\mathbf{1}_{\chi}\) is the characteristic function of \(\chi\) and \(B = \mathbf{1}_{\chi}\) (multiplication operator).
Let us consider the functionals
\begin{equation}
\psi_1(u) := \int_{\Omega} (e^u - u - 1) dx, \quad \forall u \in H_0^2(\Omega),
\end{equation}
and
\begin{equation}
\psi_2(u) := \frac{\nu}{q} \int_{\Omega} |u|^q dx, \quad \forall u \in H_0^2(\Omega).
\end{equation}
By Sobolev's embedding theorem \(H_0^2(\Omega) \hookrightarrow C^0(\overline{\Omega})\) for \(n \leq 3\), the functionals \(\psi_1\) and \(\psi_2\) are well-defined for \(q > 2\). The Gateaux derivatives of \(\psi_1\) and \(\psi_2\) at any point \(u \in H_0^2(\Omega)\) are
\begin{equation}
D\psi_1(u)(v) = \int_{\Omega} (e^u - 1) v \, dx, \qquad D\psi_2(u)(v) = \nu \int_{\Omega} |u|^{q-2} u v \, dx,
\end{equation}
for all \(v \in H_0^2(\Omega)\). Moreover, as in the general setting, \(\psi_1\) satisfies (A1)-(A3) under the smallness condition on the initial data, while \(\psi_2\) satisfies (A1)-(A3).

{Thus, Theorem \ref{last_visco} ensures
	that for sufficiently small initial data
	solutions to \eqref{plate} are globally defined and satisfy an exponential decay estimate provided that \eqref{visco_cdt_alpha_2} is satisfied. }

\end{document}